\documentclass[11pt]{article}
\usepackage{amsmath,amsthm,amssymb,amscd,fancybox,ifthen,float,epsfig,subfigure}
\usepackage[all]{xy}

\floatstyle{plain}
\restylefloat{figure}

\newcommand\DS{\mathcal{D}}

\newcommand\lot{\operatorname{l.o.t.}}

\newcommand\tot{\operatorname{tot}}
\newcommand\In{\operatorname{in}}

\newcommand\Sc{\operatorname{sc}}

\newcommand\cT{\mathcal T}
\newcommand\cM{\mathcal M}
\newcommand\cN{\mathcal N}
\newcommand\cH{\mathcal H}

\newcommand\cU{\mathcal U}

\newcommand\cQ{\mathcal Q}

\newcommand\balpha{\boldsymbol \alpha}

\newcommand\btheta{\boldsymbol \theta}

\newcommand\bEta{\boldsymbol \eta}
\newcommand\tbEta{\widetilde{\boldsymbol \eta}}
\newcommand\tepsilon{\widetilde{\epsilon}}
\newcommand\tmu{\widetilde{\mu}}
\newcommand\bBeta{\boldsymbol \beta}
\newcommand\bj{\boldsymbol j}

\newcommand\bk{\boldsymbol k}
\newcommand\bm{\boldsymbol m}

\newcommand\bn{\boldsymbol n}

\newcommand\bx{\boldsymbol x}
\newcommand\by{\boldsymbol y}
\newcommand\hbx{\hat{\boldsymbol x}}

\newcommand\bH{\boldsymbol H}

\newcommand\bE{\boldsymbol E}

\newcommand\bN{\boldsymbol N}

\newcommand\cF{\mathcal F}

\newcommand\bxi{\boldsymbol \xi}
\newcommand\tbxi{\widetilde{\bxi}}
\newcommand\bXi{\boldsymbol \Xi}

\newcommand{\THME[1]}{THME(#1)}

\newcommand\cC{\mathcal{C}}

\newcommand\cD{\mathcal{D}}

\renewcommand\Re{\operatorname{Re}}
\renewcommand\Im{\operatorname{Im}}

\newcommand\bbR{\mathbb R}

\newcommand\pa{\partial}

\newcommand\restrictedto{\upharpoonright}

\newcommand\CI{{\mathcal C}^{\infty}}

\newcommand\Id{\operatorname{Id}}

\DeclareMathOperator{\dR}{dR}

\newtheorem{theorem}{Theorem}

\theoremstyle{definition}
\newtheorem{definition}{Definition}

\theoremstyle{remark}
\newtheorem{remark}{Remark}

\begin{document}

\title{Debye Sources and the Numerical Solution\\
of the Time Harmonic Maxwell Equations, II} 

\author{Charles L. Epstein,\footnote{
    Depts. of Mathematics and Radiology, University of Pennsylvania,
    209 South 33rd Street, Philadelphia, PA 19104. E-mail:
    {cle@math.upenn.edu}.
    Research partially supported by NSF grants
    DMS06-03973, DMS09-35165 and DARPA grant HR0011-09-1-0055.} \,\,
 Leslie Greengard,\footnote{Courant Institute,
    New York University, 251 Mercer Street, New York, NY 10012.
    E-mail: {greengard@cims.nyu.edu}. 
    Research partially supported by the U.S. Department of Energy under
    contract DEFG0288ER25053 and by the Air Force Office of Scientific Research
    under MURI grant FA9550-06-1-0337 and NSSEFF Program Award 
    FA9550-10-1-0180.} \\
and Michael O'Neil\footnote{Courant Institute,
    New York University, 251 Mercer Street, New York, NY 10012.
    E-mail: {oneil@cims.nyu.edu}. 
    Research partially supported by the Air Force Office of Scientific Research
    under NSSEFF Program Award FA9550-10-1-0180.
    }}  \date{May 16, 2011}

\maketitle

\begin{abstract} In this paper, we develop a new integral representation for the
solution of the  time harmonic Maxwell equations in media with piecewise constant 
dielectric permittivity and magnetic permeability in $\bbR^3.$ 
This representation leads to a coupled system of Fredholm integral equations of the
 second kind for four scalar densities supported on the material interface. 
 Like the classical M\"{u}ller equation, it has no spurious resonances.
 Unlike the classical approach, however, the representation does not  suffer from
 low frequency breakdown.  We illustrate the performance of the method with numerical examples.
\end{abstract}

\section{Introduction}
\label{sec:uniq}
In our previous paper~\cite{EpGr2}, we introduced a new representation
for the time harmonic Maxwell equations in $\bbR^3$, based on
two scalar densities defined on the surface $bD$ of a smooth bounded region 
$D$. This bounded region may have several components, but we assume that 
its complement $D^c$ is connected. 
We refer to these densities as generalized Debye sources, since they generalize
to arbitrary geometries the classical formalism of Lorenz, Debye and Mie 
that is limited to the sphere.
We also showed in~\cite{EpGr2} that the problem of scattering 
from a perfect conductor can be reduced to the solution of a coupled
pair of Fredholm boundary integral equations of the second kind.
This system of equations was shown to be
invertible for all non-zero wave numbers in the closed upper half plane. 
Moreover, in the case that all components of $bD$ are simply connected, 
we showed that this system of equations does not suffer from 
a phenomenon called ``low frequency breakdown". 

Here, we develop an integral equation for the case of 
dielectric (interface) boundary conditions and extend the analysis of 
low frequency breakdown to the multiply connected case.  We  use,
almost exclusively, the representation of fields in the
language of forms, see~\cite{EpGr2}. We let 
\begin{equation}
\bXi(x,t)=\bxi(x)e^{-i\omega t}\quad
\bN(x,t)=\bEta(x)e^{-i\omega t},
\end{equation}
where $\bXi$ is a 1-form representing the electric field and $\bN$ is 
a 2-form representing the magnetic field, that is $\bE\leftrightarrow\bxi$ and
$\bH\leftrightarrow\bEta.$ Faraday's law and Ampere's law (the curl equations)
take the form:
\begin{equation}
  \label{eq:ME1}
  d\bxi=i\omega\mu\bEta\quad d^*\bEta=-i\omega\epsilon\bxi.
\end{equation}
For $\omega\neq 0,$ these equations imply the divergence equations, 
which take the form:
\begin{equation}
  \label{eq:ME2}
  d^*\bxi=0\quad d\bEta=0.
\end{equation}
Together, \eqref{eq:ME1} and~\eqref{eq:ME2} give an elliptic system for the pair
$(\bxi,\bEta).$ In the dielectric/interface case, our representation for
the \THME[$k$] applies {\em within} the bounded components of 
$D$ as well as in the exterior domain. The permittivity $\epsilon$ 
and permeability $\mu$ inside each component may be distinct from the
corresponding values in the exterior. 

More precisely, following the discussion in~\cite{muller},
we assume that in each component of
$\bbR^3\setminus bD,$ the EM-parameters are piecewise constant with
\begin{equation}\label{eqn4.0}
\epsilon=\tepsilon+i\frac{\sigma}{\omega}\quad 
\mu=\tmu+i\frac{\sigma'}{\omega}.
\end{equation}
Here, $\sigma$ and $\sigma'$ are non-negative numbers, 
while $\tepsilon$ and $\tmu$
are positive numbers. The complex numbers $\omega\epsilon$ and $\omega\mu$ lie
in the closed upper half plane. We, therefore, assume that the arguments
of their square roots lie in the interval $[0,\frac{\pi}{2}].$ Hence if
\begin{equation}
  k=\omega\sqrt{\epsilon\mu}=\sqrt{\omega\epsilon}\sqrt{\omega\mu},
\end{equation}
then we can assume that
\begin{equation}
  0\leq \arg k<\pi\quad\text{ provided }\quad 0\leq \arg \omega<\pi.
\end{equation}
We restrict our attention here to the standard case where 
$\omega \in \bbR^+$.
We call the system of equations~\eqref{eq:ME1} and~\eqref{eq:ME2} the
\THME[$k$].

The EM-parameters of $\Omega=\bbR^3\setminus \overline{D},$ are denoted
$(\epsilon_1,\mu_1),$ with $k_1=\omega\sqrt{\epsilon_1\mu_1}.$ If the
components of $D$ are $\{D_1,\dots,D_N\},$ then the EM-parameters for $D_j$ are
$(\epsilon_{0j},\mu_{0j}),$ with $k_{0j}=\omega\sqrt{\epsilon_{0j}\mu_{0j}}.$
The total EM-field in $\Omega$ is written in the form
\begin{equation*}
  (\bxi^{\tot},\bEta^{\tot})=(\bxi^{\Sc},\bEta^{\Sc})-(\bxi^{\In},\bEta^{\In}).
\end{equation*}
Here $(\bxi^{\In},\bEta^{\In})$ is an arbitrary solution to the \THME[$k_1$]
defined in $\Omega,$ and $(\bxi^{\Sc},\bEta^{\Sc})$ is an outgoing solution to
\THME[$k_1$], selected to insure that $(\bxi^{\tot},\bEta^{\tot})$ has the
correct behavior on $bD.$ The ${}^{\Sc}$ superscript is usually omitted in the
sequel. In the perfect conductor case, the fields in the bounded components of
$D$ are zero, while in the dielectric case the total field is each bounded
component is just a scattered field which satisfies the appropriate variant of
Maxwell's equations.

As noted, the scattered field, $(\bxi^{\Sc},\bEta^{\Sc}),$ is assumed to satisfy the
outgoing radiation condition in $\Omega.$ For the
electric field, this reads:
\begin{equation}
  \label{eq:radcond}
  i_{\hbx}d\bxi-ik\bxi=O\left(\frac{1}{|x|^2}\right),\quad
\bxi=O\left(\frac{1}{|x|}\right),
\end{equation}
where $\hbx=\frac{\bx}{\|\bx\|}.$ The same condition is satisfied by
$\star_3\bEta,$ where $\star_3$ is the Hodge star operator acting on forms
defined in $\bbR^3.$ It is a classical result that if $(\bxi,\bEta)$ solves the
\THME[$k$] for $k\neq 0$, with non-negative imaginary part, and one component
is outgoing, then so is the other. When $\omega=0$ (or $k=0$) the equations for
$\bxi$ and $\bEta$ decouple; the divergence equations,~\eqref{eq:ME2}, are no
longer a consequence of the curl equations, but are nonetheless assumed to
hold.

In this paper, we continue our study of the representation of solutions to the
time harmonic Maxwell equations in terms of the scalar Debye source densities,
which we denote by $(r,q)$, supplemented, in the case that $bD$ is of genus
$g>0,$ by the $2g$-dimensional space of harmonic 1-forms $\bj_H.$ We first show
how to use our representation to solve the time harmonic Maxwell equation when
$D$ is a dielectric with piecewise constant $\mu, \epsilon$ and $\sigma,$ under
the physical boundary condition that the tangential components of the $\bxi$
and $\bEta$ fields are continuous across $bD.$ As before we obtain a system of
Fredholm equations of second kind, which does not suffer from either spurious,
interior resonances or low frequency breakdown. The key to the good behavior as
$k\to 0,$ is that the scalar Debye sources $(r,q)$ are most directly related to
the \emph{normal} components, along $bD,$ of the $\bxi$ and $\bEta$
fields. Unlike the tangential components, which, at $k=0,$ must satisfy a
differential equation along $bD,$ ($d_{bD}\balpha=0$) the normal components are
not in any way constrained when $k=0.$ Thus, as a means of parameterizing
solutions of the \THME[$k$], the normal components behave much better as
$k\to0$ than does the tangential data.

In the second part of the paper we give a detailed analysis of the low
frequency behavior of our representation in the perfect conductor case. In the
case that the genus of $bD$ is non-zero, the dielectric problem is somewhat
simpler than the perfect conductor. At $k=0$ the solutions of the time harmonic
Maxwell equations are harmonic fields:
\begin{equation}
  d^*\bxi=d\bxi=0\text{  and  }d^*\bEta=d\bEta=0.
\end{equation}
This means that these fields represent cohomology classes in
$H^1_{\dR}(\Omega)$ and $H^2_{\dR}(\Omega),$ respectively. Therefore, in the
case of perfect conductors, the restrictions of $\bxi^+,$ and $\star_3\bEta^+$
to $b\Omega=bD$ each span subspaces of $H^1_{\dR}(bD)$ of half the total
dimension. In the dielectric case, we need to consider the jumps of these
fields across $bD.$ For all wave numbers, including $k=0,$ the harmonic
projections of these jumps span all of $H^1_{\dR}(bD).$ Hence the topological
constraint that arises in the perfect conductor problem, complicating the low
frequency behavior, is absent in the dielectric case.

In Section~\ref{s.potentials} we review the generalities of the representation
of solutions in terms of Debye sources, including the restrictions to the
boundaries and the jump conditions. In Section~\ref{s.unique} we show how to
use this approach to represent solutions to the dielectric problem and prove
the basic uniqueness results. In Section~\ref{frdeqns} we derive the boundary
integral equations for the dielectric problem and show that they do not suffer
from low frequency breakdown in either the simply connected, or non-simply
connected cases.  Finally, in Section~\ref{lfbpc}, we present a new approach to
the perfect conductor problem, when $bD$ is not simply connected. Using this
approach there is no low frequency breakdown as $k\to 0,$ and the perfect
conductor problem for the \THME[$k$]
gracefully decouples. Section~\ref{sec6} contains numerical experiments
illustrating several of the results proved in earlier sections.

\section{Debye Sources and Potentials}\label{s.potentials}

We begin by reviewing the symmetric representation of solutions to the \THME[$k$]
in terms of both potentials and anti-potentials.
We assume, as discussed above, that the time dependence is $e^{-i\omega t},$
and that
the permittivity $\epsilon,$ the permeability  $\mu,$ and the conductivity
$\sigma$ are piecewise constant. We let
$k=\omega\sqrt{\epsilon\mu},$ which we take to have non-negative imaginary
part. As in~\cite{EpGr2},
we represent the solution to the \THME[$k$] by setting:
\begin{equation}
\bxi=\sqrt{\mu}(ik\balpha-d\phi-d^*\btheta)\quad
\bEta=\sqrt{\epsilon}(ik\btheta-d^*\Psi+d\balpha),
\label{eqn29}
\end{equation}
where $\phi$ is a scalar function, $\balpha$ a one form, $\btheta$ a two form,
and $\Psi=\psi dV,$ a three form;
$\balpha$ is the usual vector potential and $\phi$ the corresponding scalar potential, while 
$\btheta$ is the vector anti-potential and $\psi$ the corresponding scalar anti-potential.
There are many possible choices for scaling
the coefficients of the various terms on the right hand side of~\eqref{eqn29}.
An advantage of the scaling in~\eqref{eqn29} is that, in each sub-region, 
$\mu$ and $\epsilon$ only appear as global multipliers, 
while the terms within the parentheses depend only on the wave
number $k=\omega\sqrt{\mu\epsilon}.$ 

Assuming that all of the potentials solve the Helmholtz equation,
$\Delta\bBeta+k^2\bBeta=0,$ where $\Delta=-(d^*d+dd^*)$ denotes the (negative)
Laplace operator in the correct form degree, for $(\bxi, \bEta)$ to satisfy the
equations in $\Gamma^c:$
\begin{equation}
  \label{eq:1.29.8.1}
  d\bxi=i\omega\mu\bEta\quad d^*\bEta=-i\omega\epsilon\bxi,
\end{equation}
it suffices to check that (in the Lorenz gauge)
\begin{equation}
d^*\balpha=-ik \phi\quad d\btheta=ik\Psi.
\label{pteq3}
\end{equation}

\subsection{Debye source representation}
We let $g_k(x,y)$ denote the fundamental solution for the scalar Helmholtz
equation, with wave number $k$, which satisfies the Sommerfeld radiation 
condition:
\[
g_k(x,y) = \frac{e^{i k |x-y|}}{4 \pi |x-y|}. \] 
For the moment we assume that
$D$ is connected, and let $\Gamma=bD.$ All of the potentials can be expressed in
terms of a pair of 1-forms $\bj, \bm$ defined on $\Gamma,$ which define
electrical and magnetic currents. As $\Gamma$ is embedded in
$\bbR^3$ these 1-forms can be expressed in terms of the ambient basis from
$\bbR^3,$ e.g.,
\begin{equation}
\bj=j_1(x)dx_1+j_2(x)dx_2+j_3(x)dx_3;
\end{equation}
we normalize with the requirement 
\begin{equation}
i_{\bn}\bj=\bj(\bn)\equiv 0.
\label{1frmnrm}
\end{equation}

It is well known \cite{Jackson,Papas} that the conditions \eqref{pteq3} are satisfied if
\begin{equation}
\begin{split}
\balpha=\int\limits_{\Gamma}g_k(x,y)[\bj(y)\cdot d\bx]dA(y)
&\quad\phi=\frac{1}{ik}\int\limits_{\Gamma} g_k(x,y) d_{\Gamma}
\star_2\bj(y) \\
\btheta=\star_3\int\limits_{\Gamma}g_k(x,y)
[\bm(y)\cdot d\bx]dA(y)&\quad 
\Psi=\frac{dV_x}{ik}\int\limits_{\Gamma} g_k(x,y) d_{\Gamma}\star_2\bm(y).
\end{split}
\label{srfint2}
\end{equation}

In the end, however, we do \emph{not} use the currents
$\bj$ and $\bm$ as the ``fundamental'' parameters. In~\cite{EpGr2}, we introduced the
notion of generalized Debye sources, $r,q$, which are scalar functions defined on $\Gamma:$
\begin{equation}
\frac{1}{ik}d_{\Gamma}\star_2\bj=rdA\quad
\frac{1}{ik}d_{\Gamma}\star_2\bm=qdA.
\label{eqn53}
\end{equation}
 From this definition, we see that $rdA$ and $qdA$ are exact
and hence their mean values vanish on $\Gamma,$
\begin{equation}
\int\limits_{\Gamma}rdA=\int\limits_{\Gamma}qdA=0
\end{equation}
This is {\bf necessary} for the conditions in~\eqref{eqn53} to hold, and thus, for
$(\bxi,\bEta)$ to satisfy the Maxwell equations. In terms of the generalized Debye sources:
\begin{equation}
\phi=\int\limits_{\Gamma} g_k(x,y)r(y)dA(y)
\quad
\Psi=dV_x\int\limits_{\Gamma} g_k(x,y) q(y)dA(y)
.\label{srfint3}
\end{equation}
Below, we derive equations for $\bj$ and $\bm$ in terms of these scalar
potentials and, if needed, harmonic 1-forms. 

Provisionally, we let $\cF^{\pm}(\epsilon,\mu,\omega,\Gamma,\bj,\bm)$ denote
the fields defined by~\eqref{eqn29} and~\eqref{srfint2}. Here $+$ refers to
the unbounded component of $\Gamma^c$ and $-,$ the bounded component. In the
sequel, $\bj$ and $\bm$ are usually taken to be functions of scalar sources,
e.g. $r$ and $q,$ via the relations in~\eqref{eqn53}, along with a possible
harmonic components $(\bj_H,\bm_H).$ The functional relationships between the
currents $(\bj,\bm)$ and the Debye sources $(r,q,\bj_H,\bm_H)$ depend on the
relationships between the currents themselves, which in turn depend on the
particulars of the boundary value problem we are trying to solve. For example,
in the perfect conductor problem, we take $\bm=\star_2\bj,$ and then
\begin{equation}
  \bj=ik[d_{\Gamma} R_0 r-\star_2d_{\Gamma}R_0q]+\bj_H,
  \label{jdefpc}
\end{equation}
where $R_0$ is the inverse of the (negative) scalar surface Laplacian,
\begin{equation}
  \Delta_{\Gamma,0}u=(\star_2d_{\Gamma}\star_2d_{\Gamma}+d_{\Gamma}\star_2d_{\Gamma}\star_2)u,
\end{equation}
 restricted to
functions of mean zero.
Note that, in the limit $\omega \rightarrow 0$, only the harmonic components of
the surface currents play a role.

\subsection{Mapping properties of the Debye source representation}
The regularity of the fields $(\bxi^{\pm},\bEta^{\pm})$ is straightforward to
describe in terms of the regularity of the Debye source data:
$(r,q,\bj_H,\bm_H).$ The harmonic components $(\bj_H,\bm_H)$ are always
infinitely differentiable. If $r$ and $q$ belong to the $L^2$-Sobolev space,
$H^s(\Gamma),$ then the currents $\bj$ and $\bm,$ defined by~\eqref{jdefpc}
belong to $H^{s+1}(\Gamma).$ As we shall see, this remains true for the
dielectric problem, even though the relationships amongst the currents and the
scalar sources are somewhat different.

For any $k$ in the closed upper half plane, and
real number $s,$ the single layer potential defines bounded maps
\begin{equation}
\begin{split}
  &S_k^{+}: H^{s}(\Gamma)\longrightarrow H^{s+\frac 32}(\Omega)\text{ and }\\
 &S_k^{-}: H^{s}(\Gamma)\longrightarrow H^{s+\frac 32}(D)
\end{split}
\end{equation}
From these observations we conclude that, if $(r,q)\in H^s(\Gamma),$ then
\begin{equation}
\begin{split}
 &\balpha^+,\btheta^+\in H^{s+\frac 52}(\Omega),\quad, \phi^{+},\Psi^+\in
  H^{s+\frac 32}(\Omega),\text{ and }\\
&\balpha^-,\btheta^-\in H^{s+\frac 52}(D),\quad, \phi^{-},\Psi^-\in
  H^{s+\frac 32}(D).
\end{split}
\end{equation}
Taken together, these observations along with~\eqref{eqn29} show that if
$(r,q)\in H^s(\Gamma),$ then
\begin{equation}
  (\bxi^{+},\bEta^{+})\in H^{s+\frac 12}(\Omega)\text{ and }
(\bxi^{-},\bEta^{-})\in H^{s+\frac 12}(D).
\end{equation}

In our applications the Debye sources $(r,q)$ are determined by solving
Fredholm equations of second kind, with kernels defined by elliptic
pseudodifferential operators of order $0.$ Hence, if the data belong to
$H^s(\Gamma),$ then so do the Debye sources, and therefore the fields in
$\Gamma^c$ belong to $H^{s+\frac 12}.$ This is precisely what one expects for
the solution of an elliptic boundary value problem for a first order elliptic
system.

\subsection{Boundary equations}
Following the convention in~\cite{EpGr2}, we use $\star_2\bxi^{\pm}_t$ and
$\star_2([\star_3\bEta]_t),$ which correspond to $\bn\times\bE$ and
$\bn\times\bH,$ respectively, to represent the tangential components, and the
scalar functions $i_{\bn}\bxi^{\pm}$ and $i_{\bn}[\star_3\bEta]$ to represent
the normal components, corresponding to $\bn\cdot\bE,$ and $\bn\cdot\bH.$ These
limiting values are given by the integral operators:
\begin{align} 
& \left(\begin{matrix}\star_2\bxi^{\pm}_t/\sqrt{\mu}\\
\star_2([\star_3\bEta^{\pm}]_t/\sqrt{\epsilon}\end{matrix}\right) =  \nonumber \\
& \frac{1}{2}\left(\begin{matrix}\pm\bm\\\mp\bj\end{matrix}\right)
+\left(\begin{matrix}-K_1& 0 & ikK_{2,t}& -K_4\\
0 & -K_1 & K_4 & ik K_{2,t}\end{matrix}\right)
\left(\begin{matrix}r\\q\\\bj\\\bm\end{matrix}\right)
\overset{d}=
  \left(\begin{matrix}\cT^{\pm}_{\bxi}\\
\cT^{\pm}_{\bEta}\end{matrix}\right)
\label{tngtds0}
\end{align}
and
\begin{align}
&  \left(\begin{matrix}i_{\bn}\bxi^{\pm}/\sqrt{\mu}\\
i_{\bn}[\star_3\bEta^{\pm}]/\sqrt{\epsilon}\end{matrix}\right)= \nonumber \\
& \frac{1}{2}\left(\begin{matrix}\pm r\\\pm q\end{matrix}\right)
+\left(\begin{matrix}-K_0& 0 & ikK_{2,n}& -K_3\\
 0 & -K_0 & K_3 & ik K_{2,n}\end{matrix}\right)
\left(\begin{matrix}r\\q\\\bj\\\bm\end{matrix}\right)
\overset{d}= \left(\begin{matrix}\cN^{\pm}_{\bxi}\\
\cN^{\pm}_{\bEta}\end{matrix}\right).
\label{nrmds0}
\end{align}
The operators $K_0, K_1, K_{2,t}, K_{2,n}, K_3$ and $K_4$ are defined in the appendix.

In what follows, we allow the bounded domain $D,$ and therefore its boundary,
$\Gamma$ to have several connected components, $\{\Gamma_1,\dots,\Gamma_N\}.$
 In all cases, we let
$\cM_{\Gamma,0}$ denote pairs of functions $(r,q)$ defined on $\Gamma$ so that
\begin{equation}
  \int\limits_{\Gamma_l}rdA=  \int\limits_{\Gamma_l}qdA=0\quad\text{ for }l=1,\dots,N.
\end{equation}
Such functions are referred to as \emph{scalar Debye sources}. We let
$\cH^1(\Gamma)$ denote the vector space of harmonic 1-forms on $\Gamma,$ that is,
the solutions of
\begin{equation}
  d_{\Gamma}\balpha=d^*_{\Gamma}\balpha=0.
\end{equation}
Let $p_l$ be the genus of $\Gamma_l,$ and
\begin{equation}
  p=p_1+\dots+p_N,
\end{equation}
the total genus of $\Gamma.$ It is a classical theorem that
\begin{equation}
  \dim\cH^1(\Gamma_l)=2p_l.
\end{equation}
Thus $\dim\cH^1(\Gamma)=2p.$

\section{Uniqueness for the Dielectric Problem}\label{s.unique}
We now apply our representation to the problem of several dielectric materials
separated by smooth bounded interfaces. We begin with the slightly simpler case
of a single connected, bounded region. The bounded region is denoted
by $D;$ we assume that $bD=\Gamma$ is connected and $\Omega$ is the exterior region
$\bbR^3\setminus \overline{D}.$ We let $(\epsilon_0,\mu_0)$ denote the EM-parameters for
$D$ and $(\epsilon_1,\mu_1)$ denote the EM-parameters for $\Omega.$ Given a
frequency $\omega \neq 0$ from the closed upper half plane, we set
\begin{equation}
  k_l=\omega\sqrt{\mu_l\epsilon_l}\text{ for }l=0,1,
\end{equation}
where $0\leq\arg k_l<\pi.$ We denote by 
$(\bxi^{+,1},\bEta^{+,1})$ the electromagnetic field in the exterior region $\Omega$ corresponding
to the exterior parameters and by
$(\bxi^{-,0},\bEta^{-,0})$ the electromagnetic field in the interior region $D$ corresponding to the 
interior parameters.

The dielectric problem for $(\bxi,\bEta)$ involves the determination of these fields satisfying
\begin{equation}\label{diebc}
  \bxi^{+,1}_t-\bxi^{-,0}_t=\bj^{\In}_t\text{ and }
(i_{\bn}\bEta^{+,1})_t-(i_{\bn}\bEta^{-,0})_t=-\star_2\bm^{\In}_t,
\end{equation}
where $(\bj^{\In},\bm^{\In})$ are two $1$-forms specified along $\Gamma$. Where
we recall that $\bn$ is the outward normal,
relative to $D$ along $\Gamma,$ and that:
\begin{equation}
  \star_2(i_{\bn}\bEta)=(\star_3\bEta)_t.
\end{equation}
Note that, in the case of a scattering problem, $(\bj^{\In},\bm^{\In})$ are the
tangential components of the known incoming field.

For $l=0,1,$ we suppose that these solutions are defined by currents
$(\bj_{l},\bm_{l})$ defined on $\Gamma.$  For
$l=0,1$  we let 
\begin{equation}\label{eqn20}
(\bxi^{\pm,l},\bEta^{\pm,l})=\cF^{\pm}(\epsilon_l,\mu_l,\omega,\Gamma,\bj_l,\bm_l)
\end{equation}
denote the solutions to the \THME[$k_l$] in $\bbR^3\setminus\Gamma$ specified
by these sources. 
Solutions defined by this ansatz automatically satisfy the outgoing radiation condition.

We seek a  solution of the dielectric problem  given by 
\begin{equation}\label{diepr1comp}
(\bxi,\bEta)=\begin{cases}
\cF^{-}(\epsilon_0,\mu_0,\omega,\Gamma,\bj_0,\bm_0)&\text{ in }D\\
\cF^{+}(\epsilon_1,\mu_1,\omega,\Gamma,\bj_1,\bm_1)&\text{ in }\Omega,
\end{cases}
\end{equation}
according to~(\ref{eqn29}),~(\ref{srfint2}).
Note, however, that we have four unknown $1$-forms on $\Gamma$ but only
two $1$-forms as data. 
We therefore, suppose
that the inner and outer currents are related by a
transformation of the form
\begin{equation}\label{crtsrc}
   \bj_{0}=\sqrt{\frac{\epsilon_1}{\epsilon_0}}
\cU\bj_{1}\text{ and }\bm_0=\sqrt{\frac{\mu_1}{\mu_0}}\cU\bm_{1},
\end{equation}
where $\cU$ is what we  refer to as a \emph{clutching map}.

\begin{definition}\label{clmp}
A clutching map is a linear isomorphism
$$\cU:\cC^0(bD;\Lambda^1)\to\cC^0(bD;\Lambda^1),$$ 
that satisfies the conditions:
\begin{enumerate}
\item The map is complex symplectic: for $\balpha, \bBeta\in
  \cC^0(bD;\Lambda^1)$ we have
  \begin{equation}\label{symplmap0}
    \oint\limits_{bD}\overline{\balpha}\wedge\bBeta=
 \oint\limits_{bD}\overline{\cU\balpha}\wedge \cU\bBeta.
  \end{equation}
\item The map $\cU$ preserves the harmonic forms $\cU\cH^1(bD)=\cH^1(bD).$
\item On the orthogonal complement of the harmonic 1-forms,
  $$\cH^1(bD)^{\bot}=d_{bD}\cC^1(bD)\oplus d_{bD}^*\cC^1(bD;\Lambda^2),$$ 
  $\cU$ reduces to the Hodge star-operator:
  \begin{equation}
\cU\restrictedto_{\cH^1(bD)^{\bot}}=\star_2.
  \end{equation}
\end{enumerate}
\end{definition}
\noindent
The simplest example is to set:
\[
 \cU \bj=\star_2\bj,
\]
For this choice,
however, a mild form of low frequency breakdown occurs in certain
non-generic cases. 

The harmonic 1-forms are invariant under the action of $\star_2.$ It is also
the case that if $\balpha\in\cH^1(bD)$ and $\bBeta=d_{bD}
u+d_{bD}^*vdA,$ then
\begin{equation}
  \int\limits_{bD}\overline{\balpha}\wedge\bBeta=
\int\limits_{bD}\overline{\balpha}\wedge\star_2\bBeta=0.
\end{equation}
From these observations and condition 3, it is apparent that $\cU$ is nothing
more than a choice of hermitian symplectic isomorphism from $\cH^1(bD)$ to
itself. If $D$ has multiple components, $\{D_1,\dots,D_N\},$ then this is done
one component at a time, i.e. we choose maps $\cU_l:\cH^1(bD_l)\to\cH^1(bD_l),$
for $l=1,\dots,N.$

With scalar Debye sources $(r_l,q_l)$ and the requirement
\begin{equation}
  d_{\Gamma}\star_2\bj_{l}=ik_lr_ldA\text{ and }d_{\Gamma}\bm_{l}=ik_lq_ldA,
\end{equation}
conditions 2 and 3 on the map $\cU$ imply that:
\begin{equation}
  \Delta_1\bj_1=i\omega(\sqrt{\mu_1\epsilon_1}d_{\Gamma}r_1-
\epsilon_0\sqrt{\frac{\mu_0}{\epsilon_1}}
\star_2d_{\Gamma}r_0)
\end{equation}
\begin{equation}
\Delta_1\bm_1=i\omega(\sqrt{\mu_1\epsilon_1}d_{\Gamma}q_1-
\mu_0\sqrt{\frac{\epsilon_0}{\mu_1}}\star_2d_{\Gamma}q_0).
\end{equation}
This means that we can define $\bj_1$ and $\bm_1$ much as in the perfect
conductor case:
\begin{equation}\label{j1r1r0}
  \bj_1=i\omega(\sqrt{\mu_1\epsilon_1}d_{\Gamma}R_0r_1-\epsilon_0\sqrt{\frac{\mu_0}{\epsilon_1}}
\star_2d_{\Gamma}R_0r_0)+\bj_H,
\end{equation}
and
\begin{equation}\label{m1q1q0}
\bm_1=i\omega(\sqrt{\mu_1\epsilon_1}d_{\Gamma}R_0q_1-\mu_0\sqrt{\frac{\epsilon_0}{\mu_1}}
\star_2d_{\Gamma}R_0q_0)+\bm_H.
\end{equation}
Here $R_0$ is the partial inverse of the (negative) scalar Laplace operator on $\Gamma$
and $\bj_H,\bm_H$ are the harmonic projections of $\bj_1,\bm_1,$ respectively.
Clearly, $\bj_1$ and $\bm_1$ are
of order $-1$ in terms of the scalar sources. If $\omega=0$, then $\bj_1$ and $\bm_1$
are purely harmonic. 

In Theorem \ref{thm01} below, we  also make use of the jump relations
\begin{equation}
  \bxi^{+,l}_t-\bxi^{-,l}_t=-\sqrt{\mu_l}\star_2\bm_l\text{ and }
(i_{\bn}\bEta^{+,l})_t-(i_{\bn}\bEta^{-,l})_t=\sqrt{\epsilon_l}\bj_l,
\label{jumpconds}
\end{equation}
which follow for $l=0,1$ from equation~\eqref{tngtds0}.
Following M\"uller's argument, we now show that our parametrization satisfies a
basic uniqueness requirement for all non-zero frequencies in the closed upper half plane.

\begin{theorem}\label{thm01} Assume that $\omega\neq 0$, that 
$\omega\mu_0,\omega\epsilon_0,\omega\mu_1,\omega\epsilon_1$
have non-negative imaginary parts, and that
  $\Re(\mu_0/\epsilon_0)>0.$ 
 Let $bD$ be connected and
  $$\cD=(r_0,q_0,r_1,q_1,\bj_H,\bm_H)$$ 
be Debye source data defining solutions to
  \THME[$k_l$] in $bD^c,$ with
  \begin{equation}\label{bcrtrl}
    \bj_0=\sqrt{\frac{\epsilon_1}{\epsilon_0}}\cU\bj_1\text{ and }
\bm_0=\sqrt{\frac{\mu_1}{\mu_0}}\cU\bm_1,
  \end{equation}
for $\cU$ a clutching map.
 If $(\bxi,\bEta)$ given by~\eqref{diepr1comp} satisfies:
 \begin{equation}
   \bxi^+_t-\bxi^-_t=0\text{ and }(i_{\bn}\bEta^+)_t-(i_{\bn}\bEta^-)_t=0,
 \end{equation}
 then the data, $(r_0,q_0,r_1,q_1,\bj_H,\bm_H),$
  are also zero.
\end{theorem}
\begin{remark} For non-zero frequencies, the representation is unique for any
  choice of clutching map. After proving this
  theorem we consider what happens if $D$ has several components, and finally
  what happens when $\omega=0.$
\end{remark}
\begin{proof} Since $(\bxi^{+,1},\bEta^{+,1})$ is an outgoing solution to
  \THME[$k_1$] it follows from M\"uller's uniqueness theorem (Theorem 61
  in\cite{muller}) that $(\bxi^{+,1},\bEta^{+,1})=(0,0)$ and
  $(\bxi^{-,0},\bEta^{-,0})=(0,0).$ To prove the theorem we need to show that
  \begin{equation}
(\tbxi,\tbEta)=\begin{cases}&(\bxi^{+,0},\bEta^{+,0})\text{ in }\Omega\\
&(\bxi^{-,1},\bEta^{-,1})\text{ in }D.
\end{cases}
  \end{equation}
   also vanishes. The
  jump conditions \eqref{jumpconds}  imply that 
  \begin{equation}
    \bxi^{+,0}_t=-\sqrt{\mu_0}\star_2\bm_0,\quad (i_{\bn}\bEta^{+,0})_t=\sqrt{\epsilon_0}\bj_0,
  \end{equation}
and
 \begin{equation}
    \bxi^{-,1}_t=\sqrt{\mu_1}\star_2\bm_1,\quad (i_{\bn}\bEta^{-,1})_t=-\sqrt{\epsilon_1}\bj_1.
  \end{equation}
The boundary current relation~\eqref{bcrtrl} then implies that
\begin{equation}\label{bcpl0}
   \bxi^{+,0}_t=-\sqrt{\mu_1}\star_2\cU\bm_1,\quad
   (i_{\bn}\bEta^{+,0})_t=\sqrt{\epsilon_1}
\cU\bj_1.
\end{equation}

Stokes theorem shows that:
\begin{equation}
\begin{split}
  \int\limits_{D}d(\overline{\bxi^{-,1}}\wedge\star_3\bEta^{-,1})=&
\oint\limits_{bD}\overline{\bxi^{-,1}}\wedge\star_3\bEta^{-,1}\\
=&\oint\limits_{bD}\overline{\bxi^{-,1}_t}\wedge\star_2[i_{\bn}\bEta^{-,1}]_t\\
=&-\sqrt{\bar{\mu}_1\epsilon_1}\oint\limits_{bD}\star_2\overline{\bm_1}\wedge\star_2\bj_1.
\end{split}
\end{equation}
We use $\oint_{bX}$ to emphasize that this is the pairing of a 2-form with a
2-cycle, where $bX$ is oriented as the boundary of $X.$ If
$\Omega_R=\Omega\cap B_R(0),$ then a second application of Stokes theorem gives
\begin{equation}
\begin{split}
  \int\limits_{\Omega_R}d(\overline{\bxi^{+,0}}\wedge\star_3\bEta^{+,0})=&
\oint\limits_{b\Omega}\overline{\bxi^{+,0}}\wedge\star_3\bEta^{+,0}+
\oint\limits_{bB_R}\overline{\bxi^{+,0}}\wedge\star_3\bEta^{+,0}\\
=&
\oint\limits_{b\Omega}\overline{\bxi^{+,0}_t}\wedge\star_2[i_{\bn}\bEta^{+,0}]_t+
\oint\limits_{bB_R}\overline{\bxi^{+,0}}\wedge\star_3\bEta^{+,0}\\
=&-\sqrt{\bar{\mu}_1\epsilon_1}\oint\limits_{b\Omega}\star_2\overline{\cU\bm_1}\wedge
\star_2\cU\bj_1
+ \oint\limits_{bB_R}\overline{\bxi^{+,0}}\wedge\star_3\bEta^{+,0}.
\end{split}
\end{equation}
The orientation of $\Gamma$ as the boundary of $\Omega$ is opposite to that as
the boundary of $D,$ and $\star_2$ is a point-wise isometry,
hence~\eqref{symplmap0} shows that these boundary contributions are of equal
magnitude, but of opposite signs.

Using the equations satisfied by these fields we also obtain that:
\begin{equation}
   \int\limits_{D}d(\overline{\bxi^{-,1}}\wedge\star_3\bEta^{-,1})=
\int\limits_{D}[i\omega\epsilon_1(\overline{\bxi^{-,1}}\wedge\star_3\bxi^{-,1})-
i\overline{\omega\mu_1}(\overline{\bEta^{-,1}}\wedge\star_3\bEta^{-,1})].
\end{equation}
and
\begin{equation}
   \int\limits_{\Omega_R}d(\overline{\bxi^{+,0}}\wedge\star_3\bEta^{+,0})=
\int\limits_{\Omega_R}[i\omega\epsilon_0(\overline{\bxi^{+,0}}\wedge\star_3\bxi^{+,0})-
i\overline{\omega\mu_0}(\overline{\bEta^{+,0}}\wedge\star_3\bEta^{+,0})].
\end{equation}
Adding, we obtain that
\begin{equation}
  \int\limits_{B_R}[i\omega\epsilon(\overline{\tbxi}\wedge\star_3\tbxi)-
i\overline{\omega\mu}(\overline{\tbEta}\wedge\star_3\tbEta)]=
\oint\limits_{bB_R}\overline{\bxi^{+,0}}\wedge\star_3\bEta^{+,0}.
\end{equation}
The real part of the left hand side is non-positive as $\Re(i\omega\epsilon)$
and $-\Re(i\overline{\omega\mu})$ are non-positive and both
$\overline{\tbxi}\wedge\star_3\tbxi$ and $\overline{\tbEta}\wedge\star_3\tbEta$ are
point-wise non-negative. On the other hand we see that the radiation condition
implies that the integral over the sphere can be written as
\begin{equation}
  \oint\limits_{bB_R}\overline{\bxi^{+,0}}\wedge\star_3\bEta^{+,0}=
\overline{\sqrt{\frac{\mu_0}{\epsilon_0}}}\oint\limits_{bB_R}
\overline{i_{\hbx}\bEta^{+,0}}
\wedge\star_3\bEta^{+,0}+o(1).
\end{equation}
As $R$ tends to infinity, the real part of the right hand side tends to a non-negative
number. 

Arguing as in M\"uller, we easily conclude that $(\bxi^{+,0},\bEta^{+,0})$
vanishes identically in $\Omega.$ Equation~\eqref{bcpl0} then implies that $\bj_1$
and $\bm_1$ are also zero. Finally, equations~\eqref{j1r1r0},~\eqref{m1q1q0}, and the
orthogonality of the Hodge decomposition imply that the Debye sources,
$(r_0,q_0,r_1,q_1,\bj_H,\bm_H),$ are themselves zero.
\end{proof}

Now suppose that $D=D_1\sqcup\dots\sqcup D_N,$ has $N>1$ connected
components. In the component, $D_l,$ we have EM parameters
$(\epsilon_{0l},\mu_{0l}).$ On each boundary component we have (inner and
outer) currents $(\bj_{1l},\bm_{1l},\bj_{0l},\bm_{0l});$ on $\Gamma_l$ they
satisfy the relations:
\begin{equation}\label{crtsrcl}
  \bj_{0l}=\sqrt{\frac{\epsilon_1}{\epsilon_{0l}}}
\cU_l\bj_{1l}\text{ and
}\bm_{0l}=\sqrt{\frac{\mu_1}{\mu_{0l}}}\cU_l\bm_{1l}
\quad l=1,\dots,N,
\end{equation}
where $\cU_l:\cC^0(bD_l;\Lambda^1)\to\cC^0(bD_l;\Lambda^1)$ is a choice of
clutching map satisfying the conditions in Definition~\ref{clmp}.  These boundary currents
are therefore defined by Debye sources
\begin{equation}
\DS_l=(r_{0l},q_{0l},r_{1l},q_{1l},\bj_{\bH l},\bm_{\bH l}),\quad l=1,\dots,N,
\end{equation}
via relations analogous to~\eqref{j1r1r0} and~\eqref{m1q1q0}:
\begin{equation}\label{dbsrccrtn2}
  \begin{split}
    \bj_{1l}&=i\omega(\sqrt{\mu_1\epsilon_1}d_{\Gamma}R_0r_{1l}-\epsilon_{0l}\sqrt{\frac{\mu_{0l}}{\epsilon_1}}
\star_2d_{\Gamma}R_0r_{0l})+\bj_{Hl},\\
\bm_{1l}&=i\omega(\sqrt{\mu_1\epsilon_1}d_{\Gamma}R_0q_{1l}-\mu_{0l}\sqrt{\frac{\epsilon_{0l}}{\mu_1}}
\star_2d_{\Gamma}R_0q_{0l})+\bm_{Hl}.
  \end{split}
\end{equation}
Now that we have introduced the relations between the Debye sources, and
boundary currents, it is useful to modify the notation introduced
in~\eqref{eqn20} so that the fields depend explicitly on this data. We use
$\cD_{l,1}$ to denote the currents $(\bj_{1l},\bm_{1l})$ defined on $\Gamma_l$
by~\eqref{dbsrccrtn2}, and $\cD_{l,0}$ the currents, $(\bj_{0l},\bm_{0l})$
defined on this surface by~\eqref{dbsrccrtn2} and~\eqref{crtsrcl}. For
$\omega\neq 0,$ we can identify
\begin{equation}
\begin{split}
  \cF^{\pm}(\epsilon_1,\mu_1,\omega,\Gamma_l,\cD_{l,1})&=
\cF^{\pm}(\epsilon_1,\mu_1,\omega,\Gamma_l,\bj_{1l},\bm_{1l})\text{ and }\\
  \cF^{\pm}(\epsilon_{0l},\mu_{0l},\omega,\Gamma_l,\cD_{l,0})&=
\cF^{\pm}(\epsilon_{0l},\mu_{0l},\omega,\Gamma_l,\bj_{0l},\bm_{0l}).
\end{split}
\end{equation}
The fields $\cF^{\pm}(\epsilon_1,\mu_1,\omega,\Gamma_l,\cD_{l,1}),$ and
$\cF^{\pm}(\epsilon_{0l},\mu_{0l},\omega,\Gamma_l,\cD_{l,0})$ are smooth
functions of the Debye sources, even  as $\omega$ goes to zero. If we need to
refer to specific field components, then we use the notation $\cF^{\pm}_{\bxi}, \cF^{\pm}_{\bEta}.$

The boundary condition,~\eqref{diebc}, defining the dielectric problem  is now
assumed to hold on each boundary component. We define a solution to the problem
via the following prescription:
\begin{equation}\label{diecslnN}
  \begin{split}
(\bxi,\bEta)&=\sum_{l=1}^{N}\cF^{+}(\epsilon_1,\mu_1,\omega,\Gamma_l,\cD_{l,1})\text{
  in }\Omega\\
(\bxi,\bEta)&=\cF^{-}(\epsilon_{0l},\mu_{0l},\omega,\Gamma_l,\cD_{l,0})\text{
  in }D_l\quad l=1,\dots,N.
\end{split}
\end{equation}
The currents in these expressions are assumed to be defined in terms of the
scalar sources and harmonic 1-forms by the relations in~\eqref{dbsrccrtn2}
and~\eqref{crtsrcl}.

We can now prove the uniqueness theorem in this case.
\begin{theorem}\label{thm02} Let $bD$ have $N$ components and fix $\omega\neq
  0,$ so that
  $$\{\omega\mu_1,\omega\epsilon_1,\omega\mu_{0l},\omega\epsilon_{0l}:\:
  l=1,\dots, N\}$$  
  have non-negative imaginary parts and the ratios
  $\Re(\mu_{0l}/\epsilon_{0l})>0.$ Let $\DS_l$ denote the Debye source data
  defining solutions to \THME[$k_{0l}$] in $D_l,$ resp. \THME[$k_{1}$] in
  $\Omega$ as specified in~\eqref{diecslnN}, where the boundary currents
  satisfy~\eqref{crtsrcl} and~\eqref{dbsrccrtn2}. If $(\bxi,\bEta)$
  satisfies~\eqref{diebc} with $\bj^{\In}=\bm^{\In}=0,$ then the Debye sources,
  $\{\DS_l\}$ are also zero.
\end{theorem}
\begin{proof} As before, M\"uller's uniqueness theorem shows that
  $(\bxi,\bEta)=(0,0)$ throughout $\bbR^3.$  To show that the data
  itself is zero, we simply apply the argument in the proof of
  Theorem~\ref{thm01} to one component of $D$ at a time.  Fix an $1\leq l_0\leq
  N,$ and define
  \begin{equation}\label{flpdflds}
    \begin{split}
      (\tbxi_{l_0},\tbEta_{l_0})&=\cF^{+}(\epsilon_{l_0},\mu_{l_0},\omega,\Gamma,\cD_{l_0,0})
\quad\text{ in }\quad\bbR^3\setminus D_{l_0}\\
(\tbxi_{l_0},\tbEta_{l_0})&=
\sum_{l\neq
  l_0}^{N}\cF^{+}(\epsilon_1,\mu_1,\omega,\Gamma_l,\cD_{l_0,1})+
\cF^{-}(\epsilon_1,\mu_1,\omega,\Gamma_l,\cD_{l_0,1})\text{ in }D_{l_0}.
    \end{split}
  \end{equation}
  These fields satisfy \THME[$k_{l_0}$] in $\bbR^3\setminus D_{l_0}$ and
  \THME[$k_1$] in $D_{l_0}.$ Since $(\bxi,\bEta)$ vanish identically, by using
  the jump conditions and~\eqref{crtsrcl}, we see that, along $\Gamma_{l_0},$
  we have the boundary data:
\begin{equation}\label{bcpl00}
    \tbxi^{-}_{l_0t}=\sqrt{\mu_1}\star_2\bm_{1l_0},
\quad (i_{\bn}\tbEta^{-}_{l_0})_{t}=-\sqrt{\epsilon_1}\bj_{1l_0},
  \end{equation}
and
\begin{equation}\label{bcpl0l}
  \tbxi^{+}_{l_0t}=-\sqrt{\mu_1}\star_2\cU\bm_{1l_0},\quad (i_{\bn}\tbEta^{+}_{l_0})_{t}=\sqrt{\epsilon_1}\cU\bj_{1l_0}.
\end{equation}
By applying the integration by parts argument from the proof of
Theorem~\ref{thm01}, we can conclude that $(\tbxi_{l_0}^+,\tbEta_{l_0}^+)$ vanishes
identically. The boundary condition in~\eqref{bcpl0l} then implies that 
\begin{equation}
  (\bj_{1l_0},\bm_{1l_0})=(0,0).
\end{equation}
Repeating this for each boundary component, and applying the Hodge
decomposition completes the proof of the theorem.
\end{proof}

The naive limit of the dielectric problem, as $\omega\to 0,$ leads to a pair of
uncoupled, underdetermined problems:
\begin{equation}\label{thme0}
  d\bxi=d^*\bxi=0\text{ and }d\bEta=d^*\bEta=0,
\end{equation}
where $(\bxi,\bEta)$ satisfy~\eqref{diebc} along $bD.$ If we let $u^{\pm}$
denote a function which is harmonic in $\bbR^3\setminus bD$ such that 
\begin{equation}
  u^{+}\restrictedto_{bD}=u^{-}\restrictedto_{bD},
\end{equation}
then the 1-form defined by
\begin{equation}
  \bxi^+=du^+\text{ in }\Omega\text{ and }\bxi^-=du^-\text{ in }D,
\end{equation}
satisfies the system of equations~\eqref{thme0} and the tangential component
has no jump across $bD.$ As we shall see, our representation of solutions
suggests that, at $\omega=0,$ we should also make use of conditions on
the normal components of the form:
\begin{equation}\label{nrmbc0}
  \epsilon_0\bxi^-_n-\epsilon_1\bxi^+_n=f\text{ and }
\mu_0(\star_3\bEta^-)_n-\mu_1(\star_3\bEta^+)_n=h.
\end{equation}
If $\omega\neq 0,$ then a condition of this type is a
consequence of~\eqref{diebc} and the Maxwell equations.
In the case of a scattering problem,
\begin{equation}\label{nrmbc1}
 f = \epsilon_1 \bj^{\In}_n  \text{ and }
h = \mu_1 \bm^{\In}_n,
\end{equation}
which can be applied for any $\omega$.
If we append this condition, then the uniqueness
theorem above extends to the zero frequency case. We state the result for the
$\bxi$-field, the analogue for the $\bEta$-field follows by application of
$\star_3.$
\begin{theorem}\label{thm03}
  Let $\bxi$ be an outgoing (zero-frequency) solution to~\eqref{thme0} defined
  in $\bbR^3\setminus bD$ such that, along $bD_l,$ $l=1,\dots,N,$ we have:
  \begin{equation}\label{diebc0hom}
    \bxi^+_t-\bxi^-_t=0\text{ and }\epsilon_{0l}\bxi^-_n-\epsilon_1\bxi^+_n=0.
  \end{equation}
Suppose that $\{\epsilon_1,
\epsilon_{0l}:\: l=1,\dots,N\}$ all have positive real part.
  If the normal components $\bxi^{\pm}_n$ have mean zero on every component of
  $bD,$ then the solution $\bxi$ vanishes identically.  
\end{theorem}
\begin{proof} 
For  $\balpha$ a closed 1-form along $bD,$ we let $[\balpha]_{\dR}$ denote its
  class in $H^1_{\dR}(bD).$ As $d\bxi^{\pm}=0,$ the tangential restrictions
  $\bxi^{\pm}_t$ are closed 1-forms. The tangential boundary condition
  in~\eqref{diebc0hom} implies that
  \begin{equation}
    [\bxi^+_t]_{\dR}= [\bxi^-_t]_{\dR}.
  \end{equation}
  On the other hand, the Mayer-Vietoris theorem implies that $[\bxi^+_t]_{\dR}$
  and $[\bxi^-_t]_{\dR}$ belong to complementary subspaces of $H^1_{\dR}(bD)$
  and therefore both must vanish. Hence there are functions $f^{\pm}$ defined
  on $bD$ so that $\bxi^{\pm}_t=df^{\pm}.$ We let $u^-$ denote the harmonic
  function in $D$ with $u^-\restrictedto_{bD}=f^-,$ and $u^+$ the outgoing
  harmonic function in $\Omega$ with $u^+\restrictedto_{bD}=f^+.$ The field
  $\bxi^--du^-$ has vanishing tangential components on $bD$ and is therefore
  identically zero in $D.$ The tangential components of $\bxi^+-du^+$ are zero,
  showing that this difference is a sum of Dirichlet fields. As $\bxi^+_n$ has
  vanishing mean on each component of $bD$ it follows from Theorem 5.7
  in~\cite{ColtonKress} that, in fact, $\bxi^+\equiv du^+.$ If we let
  $\Omega_R=\Omega\cap B_R(0),$ then the normal boundary condition
  in~\eqref{diebc0hom} easily implies that
  \begin{equation}
    \epsilon_1\int\limits_{\Omega_R}|du^+|^2dx+\sum_{l=1}^N
    \epsilon_{0l}\int\limits_{D_l}|du^-|^2dx=
\int\limits_{bB_R}u^+\overline{\pa_r u^+}dS.
  \end{equation}
As $u^+$ is outgoing, the limit, as $R\to\infty,$ of the integral over $bB_R$
is zero. This completes the proof that $\bxi\equiv 0.$ 
\end{proof}

The uniqueness of the representation of solutions to~\eqref{diebc0hom} at
zero frequency, via Debye source data turns out to depend on the choice of
clutching map $\cU:\cH^1(b\Omega)\to \cH^1(b\Omega).$ Of course $\cU=\star_2$
is such a map. Whether this suffices to prove uniqueness at $\omega=0$ turns
out to depend on a surprisingly subtle property of the Hodge star-operator on
$bD$ relative to the splitting of $H^1_{\dR}(bD)$ into the disjoint Lagrangian
subspaces (w.r.t. the symplectic form defined below in~\eqref{sympfrm1}) 
$$ H^1_{\dR}(bD)\simeq H^1_{\dR}(D)\restrictedto_{bD}\oplus H^1_{\dR}(\Omega)\restrictedto_{bD}.$$ 
We let $\cH^1_{D}(bD)$ denote
harmonic representatives for the image of the injective map
$H^1_{\dR}(D)\hookrightarrow H^1_{\dR}(bD),$ and $\cH^1_{\Omega}(bD)$ harmonic
representatives for
the image of $H^1_{\dR}(\Omega)\hookrightarrow
H^1_{\dR}(bD).$ As noted above, the Hodge star-operator maps harmonic forms to harmonic forms.
\begin{definition}
We say that $bD$ is \emph{H-generic} if
\begin{equation}\label{Hgencond}
  \star_2\cH^1_{\Omega}(bD)\cap\cH^1_{D}(bD)=\{0\}.
\end{equation}
\end{definition}
Evidently a simply connected manifold is H-generic. More generally,
the property of H-genericity depends only the conformal structure on $bD$
(induced from its embedding into $\bbR^3$) and the splitting
\begin{equation}
  H^1_{\dR}(bD)=H^1_{\dR}(D)\restrictedto_{bD}\oplus
H^1_{\dR}(\Omega)\restrictedto_{bD},
\end{equation}
which depends on
the isotopy class of the embedding of $bD\hookrightarrow \bbR^3.$

It is a deep theorem in algebraic geometry, stemming form work of Riemann, that
the set of H-generic structures is the complement of a real analytic
hypersurface in Teichm\"uller space, and therefore open and dense. 
On the other hand, there certainly exist surfaces for which~\eqref{Hgencond}
fails. For example if $bD$ is a torus of revolution, as described in Example
1.5 of~\cite{EpGr2}, then we have
$\star_2\cH^1_{\Omega}(bD)\cap\cH^1_{D}(bD)=\cH^1_{D}(bD).$ If we were to take
$\cU=\star_2,$ in such a case, then the Debye source
representation at zero frequency would have a   non-trivial null-space  of
dimension equal to $\dim \star_2\cH^1_{\Omega}(bD)\cap\cH^1_{D}(bD).$ This
would then lead to a mild form of low frequency breakdown. To avoid this
eventuality we need to make a different choice of $\cU.$

\begin{definition} \label{clutchdef}
A clutching map $\cU:\cH^1(b\Omega)\to \cH^1(b\Omega),$ 
 is \emph{admissible} if, for each component $bD_l$ of $b\Omega$, we
  have
  \begin{equation}
    \cU_l\star_2\cH^1_{D_l}(bD_l)\cap \star_2\cH^1_{D_l^c}(bD_l)=\{0\}.
  \end{equation}
\end{definition}
It is easy to see that such maps always exist. As we do the construction one
component at a time, we can restrict attention to the case that $bD$ is
connected. We let $\omega$ denote the hermitian symplectic form
\begin{equation}\label{sympfrm1}
  \omega(\balpha,\bBeta)=\oint\limits_{bD}\overline{\balpha}\wedge\bBeta.
\end{equation}
This is a non-degenerate pairing that is well defined on $H^1_{\dR}(bD).$
Moreover,
\begin{equation}
  \omega(\star_2\balpha,\star_2\bBeta)=\omega(\balpha,\bBeta).
\end{equation}
The topological interpretation of this pairing implies that both
$\star_2\cH^1_{D}(bD)$ and $\star_2\cH^1_{D^c}(bD)$ are Lagrangian
subspaces; that is, $\omega$ restricted to these subspaces is identically zero. In fact
$\cH^1_{D}(bD)$ is a Lagrangian subspace complementary to $\star_2\cH^1_{D}(bD),$
and these subspaces are complexifications of real Lagrangian subspaces. The map
$\cU_l=\Id$ on $\cH^1(bD_l)$ is always an admissible clutching map. For if
\begin{equation}
  \balpha\in \star_2\cH^1_{D_l}(bD_l)\cap \star_2\cH^1_{D_l^c}(bD_l),
\end{equation}
then
\begin{equation}
  \star_2 \balpha\in \cH^1_{D_l}(bD_l)\cap \cH^1_{D_l^c}(bD_l)=\{0\}.
\end{equation}

\begin{theorem}\label{thm04}
  Suppose that every component of $bD$ is $H$-generic, in which case we
  take $\cU=\star_2,$ or that $\cU$ is an admissible clutching map.  If an
  outgoing (zero-frequency) solution to~\eqref{thme0} in $\bbR^3\setminus bD,$ defined by Debye
  source data $\{\cD_l\}$ via~\eqref{diecslnN}, vanishes, then the data
  defining it vanishes as well.
\end{theorem}
\begin{proof}
Note first that, at zero frequency, the Debye source data uncouples, with $\bxi$ determined by  
$$\{(r_{0l},r_{1l},\bm_{Hl}):\: l=1,\dots, N\}\subset \cM_{\Gamma,0}\times\cH^1(bD)$$ 
via~\eqref{diecslnN},  and $\bEta$ determined by 
$$\{(q_{0l},q_{1l},\bj_{Hl}):\: l=1,\dots, N\}\subset \cM_{\Gamma,0}\times\cH^1(bD).$$ 
  We restrict our attention here to the electric field and suppose that the Debye source data
  determines a  1-form $\bxi$ that vanishes
  identically. First we assume that the harmonic 1-forms $\{\bm_{Hl}:\:
  l=1,\dots,N\}$ are all zero. This is always the case if $bD$ is simply
  connected.  In $D_l$ we have
\begin{equation}
  \bxi^-=-\sqrt{\mu_{0l}}d\phi_l^-,
\end{equation}
where
\begin{equation}
  \phi_l=\int\limits_{bD}g_0(x,y)r_{0l}dA.
\end{equation}
The fact that $d\phi_l^-=0$ implies that $\phi_l^-$ is constant, which implies
that $\phi_l^+,$ its continuation to $D_l^c,$ is also constant along $bD_l.$
Integrate by parts, using that $\pa_{\bn}\phi^+=r_{0l},$ and $\phi_l=c$ a
constant, to obtain
\begin{equation}
  \int\limits_{D_l^c}|\nabla\phi_l^+|^2dx=c\int\limits_{bD_l}r_{0l}.
\end{equation}
Either because $c=0$ or because $r_{0l}$ has mean zero we see that, in fact,
$\phi_l^+\equiv 0$ as well, which implies that $r_{0l}=0.$
 Similarly, in $\Omega$ we see that the harmonic function
\begin{equation}
  \phi_1=\sum_{l=1}^N\int\limits_{bD_l}g_0(x,y)r_{1l}dA
\end{equation}
is constant and therefore $0.$ This function is continuous across each boundary
component, so it is also zero in each component of $D_{l}.$ The jump in
the normal derivative of $\phi_l$ across $bD_l$ is $r_{1l},$  which must
also vanish. This completes the case where $\bm_H\equiv 0.$

It remains only to show that the harmonic 1-forms $\{\bm_{Hl}\}$ vanish. For each $1\leq
l_0\leq N$ we use the prescription in~\eqref{flpdflds} to define a harmonic
field in $\bbR^3\setminus bD_{l_0}.$ The relations in~\eqref{bcpl00}
and~\eqref{bcpl0l} show that along $bD_{l_0}$
\begin{equation}
  \tbxi^{-}_{l_0t}=\sqrt{\mu_{1}}\star_2\bm_{Hl_0}\text{ and }
\tbxi^{+}_{l_0t}=-\sqrt{\mu_{1}}\star_2\cU\bm_{Hl_0},
\end{equation}
where $\bm_{Hl_0}$ is a harmonic 1-form on $bD_{l_0}.$ As $d\tbxi_{l_0}=0$ in
$\bbR^3\setminus bD_{l_0},$ we see that
\begin{equation}
  \star_2\bm_{Hl_0}\in \cH^1_{D_{l_0}}(bD_{l_0})\text{ and }
\star_2\cU\bm_{Hl_0}\in \cH^1_{D_{l_0}^c}(bD_{l_0}).
\end{equation}
The assumption that $\cU$ is an admissible clutching map implies that
$\star_2\bm_{Hl_0}=0,$ 
completing the proof.
\end{proof}

To conclude this discussion, we show that in the case that $bD$ is not
H-generic, and we take $\cU=\star_2,$ then the Debye source representation has a
non-trivial nullspace. Suppose that $bD$ is connected and
\begin{equation}\label{nongen0}
\star_2\bm_{H}\in
\star_2\cH^1_{\Omega}(bD)\cap\cH^1_{D}(bD).
\end{equation}  We let $\tbxi_j$, $j=0,1,$ denote the fields
defined in $\bbR^3\setminus bD$ by
\begin{equation}
\begin{split}
  \bxi_1^{\pm}&=\cF_{\bxi}^{\pm}(\epsilon_1,\mu_1,0,bD,\{0,0,\bm_H\})\\
  \bxi_0^{\pm}&=\cF_{\bxi}^{\pm}(\epsilon_0,\mu_0,0,bD,\{0,0,\star_2\bm_H\}).
\end{split}
\end{equation}
The assumption in~\eqref{nongen0} implies that
$\star_2\bm_{H}\in\cH^1_{D}(bD),$ and $\bm_H\in\cH^1_{\Omega}(bD).$ Arguing as
in the proof of Proposition 7.12 in~\cite{EpGr2}, we see that the cohomology
classes $[\bxi_1^+]$ and $[\bxi_0^-]$ are trivial. Thus we can find
\emph{unique} scalar Debye sources $(r_0,r_1)$ defining fields
$d\phi_{j}^{\pm}$ so that
\begin{equation}
  d\phi_1^++\bxi_1^+=0 \text{ and }d\phi_0^-+\bxi_0^-=0.
\end{equation}
This shows that the null-space of the Debye representation at zero frequency
includes a subspace isomorphic to $\star_2\cH^1_{\Omega}(bD)\cap\cH^1_{D}(bD).$
 From the proof of Theorem~\ref{thm04} we easily conclude that this is exactly
the null-space of the representation.

\section{Integral Equations for the Dielectric Problem}\label{frdeqns}
In this section we derive Fredholm integral equations of the second kind for
solving the dielectric problem, in terms of the Debye sources. The dielectric
interface boundary conditions are formulated in~\eqref{diebc}. We rewrite them
slightly for this section
\begin{equation}\label{diebc2}
  \star_2[\bxi^{+,1}_t-\bxi^{-,0}_t]=\star_2\bj^{\In}_t\text{ and }
\star_2[(\star_3\bEta^{+,1})_t-(\star_3\bEta^{-,0})_t]=\star_2\bm^{\In}_t.
\end{equation}
The data $(\bj^{\In},\bm^{\In})$ are arbitrary 1-forms defined on $\Gamma.$ We
usually assume that these are defined as boundary data of a solution
$(\bxi^{\In},\bEta^{\In})$ to the time harmonic Maxwell equations in $\Omega.$ To
obtain equations in terms of the Debye sources, we apply the operators
$G_0\star_2d_{\Gamma}$ and $d_{\Gamma}\star_2$ to these equations. In the simply connected
case this suffices to solve the original problem, as a 1-form $\bk$ on
$\Gamma$ is specified by the scalar functions $\star_2d_{\Gamma}\bk$ and
$d_{\Gamma}^*\bk.$ 

The tangential, and normal operators, $\cT^{\pm}_{\bxi}(k),
\cT^{\pm}_{\bEta}(k),$ and $\cN^{\pm}_{\bxi}(k), \cN^{\pm}_{\bEta}(k),$ used to
construct $\bxi^{\pm,l}$ and $\bEta^{\pm,l}$ are defined in~\eqref{tngtds0}
and~\eqref{nrmds0}.  There are differences between these operators and those
used in the perfect conductor case. For the latter case, we used \eqref{jdefpc}
to define $\bj$ in terms of of a single set of Debye sources.  Here, for $l=0,$
or $1,$ the vector sources $\bj_l,\bm_l$ in~\eqref{tngtds0} and~\eqref{nrmds0}
depend on four scalar potentials $(r_0,q_0,r_1,q_1),$ and possibly the harmonic
1-forms $\bj_H,\bm_H,$ which we abbreviate as $(r,q),$ $(r,q,\bj_H,\bm_H)$
resp. Moreover, the vector sources $\bj_1$ and $\bm_1$ are independent of one
another, with the inner and outer sources, $(\bj_0,\bm_0),$ $(\bj_1,\bm_1)$,
satisfying the relations in~\eqref{crtsrc}. Notice also that the wave numbers
are different in different regions of space, so that, e.g. we use
$\cT^{-}_{\bxi}(k_0)$ to define the $\bxi$-field in the interior of $D,$ and
$\cT^{+}_{\bxi}(k_1)$ to define the $\bxi$-field in the exterior domain
$\Omega.$ Finally, it should be noted that, when $bD$ has several components,
the dependencies between currents and scalar Debye sources stated
in~\eqref{dbsrccrtn2} are local within each boundary component. This is the
case even though the formul{\ae} in~\eqref{diecslnN} imply that the integral
equations we ultimately have to solve intertwine the different boundary
components, albeit via smoothing terms.

While our approach was motivated by that of  M\"uller \cite{muller}, it is worth noting
that there are two important
differences. He used two vector fields $(\bj_1, \bm_1)$ as unknowns combined with 
a relation of the form 
\[
   \bj_{0}=\sqrt{\frac{\epsilon_1}{\epsilon_0}} \bj_{1}\text{ and }
   \bm_0=\sqrt{\frac{\mu_1}{\mu_0}} \bm_{1}
\]
and proved well-posedness for the resulting integral equation. In order to avoid low-frequency
breakdown and achieve graceful uncoupling of the fields, the introduction of the Debye source
representation necessitates a more delicate analysis of the interior and exterior representations,
and the use, in some cases, of a non-trivial clutching map. 

\subsection{The Single Component Case}
We first write the integral equations for the case that $bD$ is connected.  In
this case  $(r_0,q_0,r_1,q_1)\in\cM_{\Gamma,0}^2$ and
$(\bj_H,\bm_H)\in\cH^1(bD)\times\cH^1(bD)$ need to be determined. When
possible, we omit the explicit mention of these arguments below. Our first two
equations are

\begin{equation}\label{diebc3}
  G_0d_{\Gamma}^*[(\bxi^{+,1}_t-\bxi^{-,0}_t)]=G_0d_{\Gamma}^*\bj^{\In}\text{ and }
 G_0d_{\Gamma}^*[(\star_3\bEta^{+,1})_t-(\star_3\bEta^{-,0})_t]= G_0d_{\Gamma}^*\bm^{\In}.
\end{equation}
More explicitly, these equations are
\begin{equation}
\begin{split}\label{tngeqns}
  G_0\star_2d_{\Gamma}[\sqrt{\mu_1}\cT^+_{\bxi}(k_1)-
\sqrt{\mu_0}\cT^-_{\bxi}(k_0)]&=G_0d_{\Gamma}^*\bj^{\In}\\
  G_0\star_2 d_{\Gamma}[\sqrt{\epsilon_1}\cT^+_{\bEta}(k_1)-
\sqrt{\epsilon_0}\cT^-_{\bEta}(k_0)]&=G_0d_{\Gamma}^*\bm^{\In}
\end{split}
\end{equation}
We use the fact that 
\begin{equation}
  G_0d_{\Gamma}^*d_{\Gamma}G_k=\frac{\Id}{4}+O(-1),
\end{equation}
to deduce that these equations are of the form:
\begin{equation}
  \begin{split}\label{ldtneq}
    \frac{-\sqrt{\mu_1}}{4}r_1+
    \frac{\sqrt{\mu_0}}{4}r_0+\lot&=G_0d_{\Gamma}^*\bj^{\In}\\
\frac{\sqrt{\epsilon_1}}{4}q_1- \frac{\sqrt{\epsilon_0}}{4}q_0+\lot&=G_0d_{\Gamma}^*\bm^{\In}.
  \end{split}
\end{equation}
Here and in the sequel, $\lot$ refers to pseudodifferential operators of
negative order (when applied to $(r,q)$).

To derive the other equations we observe that for a 1-form $\balpha$ smooth up
to $\Gamma$ we have that
\begin{equation}\label{fnct1}
  d_{\Gamma}(\balpha_t)=(d\balpha)_t,
\end{equation}
and
\begin{equation}
 d\bxi^{\pm,l}=i\omega\mu_l\bEta^{\pm,l}\text{ and }
 d\star_3\bEta^{\pm,l}=-i\omega\epsilon_l\star_3\bxi^{\pm,l}.
\end{equation}
This means that
\begin{equation}
  (d\bxi^{\pm,l})_t=i\omega\mu_l(\bEta^{\pm,l})_t\text{ and }
 (d\star_3\bEta^{\pm,l})_t=-i\omega\epsilon_l(\star_3\bxi^{\pm,l})_t.
\end{equation}
The right hand sides of these equations are essentially the normal components for
$\bEta^{\pm,l}$ and $\bxi^{\pm,l},$ respectively. On the other hand,~\eqref{fnct1} implies
that
\begin{equation}
  (d\bxi^{+,1})_t- (d\bxi^{-,0})_t=d_{\Gamma}\bj^{\In}\text{ and }
(d\star_3\bEta^{+,1})_t-(d\star_3\bEta^{-,0})_t= d_{\Gamma}\bm^{\In},
\end{equation}
and therefore it is reasonable to assume that there are
functions $f,h$ defined on $\Gamma,$ of size $O(1),$ so that
\begin{equation}
  d_{\Gamma}\bj^{\In}=i\omega fdA\text{ and }d_{\Gamma}\bm^{\In}=i\omega hdA.
\end{equation}
This is automatic if the data $(\bj^{\In},\bm^{\In})$
come from a solution of the \THME[$k_1$]. Note that $(f,h)\in\cM_{\Gamma,0}.$ 

With this understood, the second set of equations becomes:
\begin{equation}
\begin{split}\label{nrmeqns}
  \mu_1\sqrt{\epsilon_1}\cN^+_{\bEta}(k_1)-
\mu_0\sqrt{\epsilon_0}\cN^-_{\bEta}(k_0)&=f,\\
  -\epsilon_1\sqrt{\mu_1}\cN^+_{\bxi}(k_1)+
\epsilon_0\sqrt{\mu_0}\cN^-_{\bxi}(k_0)&=h.
\end{split}
\end{equation}
These equations are of the form:
\begin{equation}\label{ldtneq2}
  \begin{split}
    \left(\frac{\mu_1\sqrt{\epsilon_1}}{2}q_1+\frac{\mu_0\sqrt{\epsilon_0}}{2}q_0\right)+\lot&=f,\\
    -\left(\frac{\sqrt{\mu_1}\epsilon_1}{2}r_1+\frac{\sqrt{\mu_0}\epsilon_0}{2}r_0\right)+\lot&=h.\\
  \end{split}
\end{equation}
These relations, along with~\eqref{ldtneq} show that~\eqref{tngeqns}
and~\eqref{nrmeqns} are a Fredholm system of the second kind for
$(r_0,q_0,r_1,q_1),$ provided that
\begin{equation}
  \mu_0\mu_1\epsilon_0\epsilon_1\neq 0\text{ and }  (\mu_0+\mu_1)(\epsilon_0+\epsilon_1)\neq 0.
\end{equation}

If the genus of $\Gamma$ is $p\neq 0,$ then we also need to determine
$(\bj_H,\bm_H)\in\cH^1(\Gamma)\times\cH^1(\Gamma).$ Equations~\eqref{tngeqns}
and~\eqref{nrmeqns} do not quite suffice as the right hand sides of these
equations explicitly annihilate the harmonic projections of $\bj^{\In}$ and
$\bm^{\In}.$ We have two distinct choices as to how we should augment these
equations to capture the projection into the harmonic vector fields. On the one
hand we could use a basis  $\Psi=(\psi_1,\dots,\psi_{2g})$  for
$\cH^1(\Gamma),$ the harmonic one forms on $\Gamma,$ and then augment the
equations~\eqref{tngeqns} and~\eqref{nrmeqns} with the $4g$ equations:
\begin{equation}\label{auxeqn}
\begin{split}
 \langle [\bxi^{+,1}_t-\bxi^{-,0}_t],\Psi\rangle&=\langle\bj^{\In},\Psi\rangle\\
\langle [(\star_3\bEta^{+,1})_t-(\star_3\bEta^{-,0})_t)],\Psi\rangle&
=\langle\bm^{\In},\Psi\rangle.
\end{split}
\end{equation}

We can also follow a more traditional, and geometric approach. Let
$\{(A_j,B_j):\: j=1,\dots,g\}$ be a basis for $H_1(\Gamma),$ normalized as
above: the $A$-cycles span $[H^1(D)]^{\bot},$ and  the $B$-cycles span
$[H^1(\Omega)]^{\bot}.$ We can then augment the equations~\eqref{tngeqns}
and~\eqref{nrmeqns} with conditions on the circulations:
\begin{equation}\label{auxeqn2}
  \begin{split}
    \int\limits_{A_j}[\bxi^{+,1}_t-\bxi^{-,0}_t]&=\int\limits_{A_j}\bj^{\In}\\
    \int\limits_{B_j}[\bxi^{+,1}_t-\bxi^{-,0}_t]&=\int\limits_{B_j}\bj^{\In}\\
  \int\limits_{A_j}[(\star_3\bEta^{+,1})_t-(\star_3\bEta^{-,0})_t)]&=
\int\limits_{A_j}\bm^{\In}\\
    \int\limits_{B_j}[(\star_3\bEta^{+,1})_t-(\star_3\bEta^{-,0})_t)]&=
\int\limits_{B_j}\bm^{\In}.
  \end{split}
\end{equation}
If we have Debye source data in the null-space of~\eqref{tngeqns}
and~\eqref{nrmeqns}, then the tangential fields $[\bxi^{+,1}_t-\bxi^{-,0}_t]$
and $[(\star_3\bEta^{+,1})_t-(\star_3\bEta^{-,0})_t)]$ are harmonic 1-forms. As
$H^1_{\dR}(bD)$ and $H_1(bD)$ are dual vector spaces, via this pairing, such
data lies in the null-space of~\eqref{auxeqn} if and only if it is in the
null-space of~\eqref{auxeqn2}.

\subsection{The Multiple Component Case}
In this subsection, we assume that $bD=bD_1\cup\dots\cup bD_N.$ Taking account
of the differences between~\eqref{diepr1comp} and~\eqref{diecslnN}, we see that
the equations in this case differ somewhat from the equations in the case
where $bD$ is connected. The differences are all in the form of
smoothing operators, and therefore the computations of the leading order terms
apply {\it mutatis mutandis}. For clarity, we augment the notation for the
tangential and normal restrictions to specify a particular component of the boundary,
e.g.  it should be understood that $\cT^+_{\bxi}(k_1, bD_l)$ is a tangent field
on $bD_l$ that depends on $(r_{0l},r_{1l},q_{0l},q_{1l},\bj_{Hl},\bm_{Hl}).$
For $m\neq l$, we
use $\cF^{+,l}_{\bxi,t}(\epsilon_1,\mu_1,\omega,\Gamma_m,\cD_{m,1}),$
$\cF^{+,l}_{\bxi,n}(\epsilon_1,\mu_1,\omega,\Gamma_m,\cD_{m,1}),$  to denote
the tangential and normal components, resp., of the $\bxi$-field along
$bD_l$ of the fields defined by the sources $(\bj_{1m},\bm_{1m})$ on $bD_m.$
The notation  $\cF^{+,l}_{\bEta,t},  \cF^{+,l}_{\bEta,n}$ has the analogous
meaning for the $\bEta$-field.

With these notational conventions, we can now give the equations for the
multiple component case. We let $(\bj^{\In}_l,\bm^{\In}_l)$ denote the incoming
fields along $bD_l.$ As in the previous case we assume that these fields
arise as restrictions of a single solution $(\bxi^{\In},\bEta^{\In})$ of
\THME[$k_1$], and therefore, there are functions $(f_l,h_l)$ of size $O(1)$ so
that, along $bD_l$ we have:
\begin{equation}\label{nrmcmpmc}
  d_{\Gamma_l}\bj^{\In}_l=i\omega f_ldA\text{ and }d_{\Gamma_l}\bm^{\In}_l=i\omega h_ldA.
\end{equation}

Now fix $1\leq l\leq N.$ Along $bD_l$ the tangential equations for the solution of the
dielectric problem in terms of the Debye sources reads:

\begin{equation}
\begin{split}\label{tngeqnsmc}
  G_0\star_2d_{\Gamma}[\sqrt{\mu_1}\cT^+_{\bxi}(k_1,bD_l)-
\sqrt{\mu_{0l}}\cT^-_{\bxi}(k_{0l},bD_l)]
+& 
\\G_0d_{\Gamma}^*\Big[\sum_{m\neq l}
\cF^{+,l}_{\bxi,t}(\epsilon_1,\mu_1,\omega,&\Gamma_m,\cD_{m,1})\Big]\\
&=G_0d_{\Gamma}^*\bj^{\In}_l\\
  G_0\star_2d_{\Gamma}[\sqrt{\epsilon_1}\cT^+_{\bEta}(k_1)-
\sqrt{\epsilon_{0l}}\cT^-_{\bEta}(k_{0l})]+&\\
G_0d_{\Gamma}^*\Big[\sum_{m\neq l}
\cF^{+,l}_{\bEta,t}(\epsilon_1,\mu_1,\omega,&\Gamma_m,\cD_{m,1})\Big]\\
&=G_0d_{\Gamma}^*\bm^{\In}_l.
\end{split}
\end{equation}
Away from $\Gamma_l$ we have the relations
\begin{equation}
\begin{split}
    d\cF^{+}_{\bxi}(\epsilon_1,\mu_1,\Gamma_l,\cD_{l,1})&=
i\omega\mu_1\cF^{+}_{\bEta}(\epsilon_1,\mu_1,\Gamma_l,\cD_{l,1})\\
d\star_3\cF^{+}_{\bEta}(\epsilon_1,\mu_1,\Gamma_l,\cD_{l,1})&=
-i\omega\epsilon_1\star_3\cF^{+}_{\bxi}(\epsilon_1,\mu_1,\Gamma_l,\cD_{l,1}).
\end{split}
\end{equation}
These relations and~\eqref{nrmcmpmc} show that the  normal equations are
\begin{equation}
\begin{split}\label{nrmeqnsmc}
  \mu_1\sqrt{\epsilon_1}\cN^+_{\bEta}(k_1,bD_l)-
\mu_{0l}\sqrt{\epsilon_{0l}}\cN^-_{\bEta}(k_{0l},bD_l)&+\\
\mu_1\sum_{m\neq l}
\cF^{+,l}_{\bEta,n}(\epsilon_1,\mu_1,\omega,&\Gamma_m,\cD_{m,1})\\
&=f_l=\left(\frac{\star_2d_{\Gamma}\bj^{\In}}{i\omega}\text{ if }\omega\neq 0\right)\\
  -\epsilon_1\sqrt{\mu_1}\cN^+_{\bxi}(k_1,bD_l)+
\epsilon_{0l}\sqrt{\mu_{0l}}\cN^-_{\bxi}(k_{0l},bD_l)+\\
-\epsilon_1\sum_{m\neq l}
\cF^{+,l}_{\bxi,n}(\epsilon_1,\mu_1,\omega,&\Gamma_m,\cD_{m,1})\\
&=h_l=\left(\frac{\star_2d_{\Gamma}\bm^{\In}}{i\omega}\text{ if }\omega\neq 0\right).
\end{split}
\end{equation}
As noted above, the new terms in these equations are smoothing operators. Thus, with
the obvious changes in notation, the equations for the leading order
parts,~\eqref{ldtneq} and~\eqref{ldtneq2} apply equally well in the
multi-component case to show that these are again Fredholm equations of the second
kind for $\{(r_{0l},r_{1l},q_{0l},q_{1l}):\: l=1,\dots,N\}.$ To complete the
system we append either the $4p$ equations in~\eqref{auxeqn}, or those
in~\eqref{auxeqn2}, where $p$ is  the total genus of
$\Gamma.$

\subsection{Low Frequency Behavior}
Suppose that we are given Debye data in the null-space of the system defined
by~\eqref{tngeqnsmc},~\eqref{nrmeqnsmc}, and~\eqref{auxeqn}. If $\omega\neq 0,$
then it is apparent from the Hodge theorem that $(\bxi^{\pm},\bEta^{\pm})$
satisfy~\eqref{diebc} with $\bj^{\In}=\bm^{\In}=0.$ Hence Theorem~\ref{thm02}
shows that the Debye data is in fact zero. If we are given Debye data, so that,
at $\omega=0,$ the limiting homogeneous boundary conditions~\eqref{diebc0hom}
hold, then, provided that $\cU$ is an admissible clutching map,
Theorem~\ref{thm04} implies that again, the solution $(\bxi^{\pm},\bEta^{\pm})$
is identically zero, as is the Debye data. That the mean values of the normal
components $(i_{\bn}\bxi^+,i_{\bn}\star_3\bEta)$ vanish on each component of
$bD$ follows immediately, as this is true for all solutions defined by Debye
data.

Thus, if we make the assumption that the incoming field satisfies the
conditions in~\eqref{nrmcmpmc}, and $\cU$ is an admissible clutching map, then
this system of equations displays no low frequency breakdown. The uniqueness
results for the Debye representations, Theorems~\ref{thm02} and~\ref{thm04},
show that, using the ansatz in~\eqref{diecslnN}, the only Debye data which
leads to a solution to the homogeneous dielectric problem is
zero. As~\eqref{tngeqnsmc},~\eqref{nrmeqnsmc}, and~\eqref{auxeqn} is a system
of Fredholm equations of second kind, this proves the solvability.
\begin{theorem}
  Let $D$ be a union of smooth bounded regions in $\bbR^3$ with connected
  complement $\Omega.$  Suppose further that $\cU$ is an admissible clutching map.
\begin{enumerate}
\item If $\omega\neq 0$ and we assume that 
$$\{\omega\mu_1,\omega\epsilon_1,\omega\mu_{0l},\omega\epsilon_{0l}:\:
l=1,\dots, N\}$$ have non-negative imaginary parts, and the ratios
$\Re(\mu_{0l}/\epsilon_{0l})>0,$ then the system of
equations~\eqref{tngeqnsmc},~\eqref{nrmeqnsmc}, and~\eqref{auxeqn} for
$$\{(r_{0l},r_{1l},q_{0l},q_{1l},\bj_{Hl},\bm_{Hl}):\: l=1,\dots,N\},$$ is
solvable for any right hand side $(\bj^{\In},\bm^{\In}).$
\item If $\omega=0,$ then, assuming that, for some $\phi,$
  $\{e^{i\phi}\epsilon_1, e^{i\phi}\epsilon_{0l}:\: l=1,\dots,N\}$ all have
  positive real part, the $\bxi$-equations
  in~\eqref{tngeqnsmc},~\eqref{nrmeqnsmc}, and~\eqref{auxeqn} for the Debye
  sources $\{(r_{0l},r_{1l},\bm_{Hl}):\: l=1,\dots,N\},$ are solvable for
  arbitrary admissible data $d_{\Gamma}^*\bj^{\In}$, $\{f_1,\dots,f_N\}$ and
  $\langle \bj^{\In},\Psi\rangle.$ Here data is admissible if
  $d_{\Gamma}\bj^{\In}=0$ and
\begin{equation}
  \int\limits_{\Gamma_l}f_ldA=0.
\end{equation}
\item The analogous result holds for the $\bEta$-equations assuming that there
  is a $\theta$ so that  $\{e^{i\theta}\mu_1, e^{i\theta}\mu_{0l}:\:
  l=1,\dots,N\}$ all have positive real parts.
\end{enumerate} 
\end{theorem}
\begin{remark} The positivity conditions in this theorem are true if
  $\Im\omega\geq 0$ and the physical constants satisfy the conditions given
  in~\eqref{eqn4.0}.  As noted above, for an H-generic boundary we can let
  $\cU=\star_2.$ 
H-genericity is an open and dense
  condition on surfaces in $\bbR^3.$ 
\end{remark}

\section{Low frequency behavior of the perfect conductor problem}\label{lfbpc}
The operators defining the Hodge star
($\star_2$) of the tangential components of $(\bxi^{\pm},\star_3\bEta^{\pm})$ are given by
\begin{equation}
\begin{split}
  \cT^{\pm}_{\bxi}(k)
&=
\frac{\pm \bm}{2}-
 K_1r + ik K_{2,t}\,\bj -K_4\bm\\
 \cT^{\pm}_{\bEta}(k)
&=
\frac{\mp \bj}{2}-
 K_1q + ik K_{2,t}\,\bm +K_4\bj
\label{eqn888.1}
\end{split}
\end{equation}
The equations for the normal components of the $\bEta$-fields are
\begin{equation}
\cN^{\pm}_{\bEta}(k)
=\left(\frac{\pm\Id}{2}-K_0\right)q+
K_3 \,\bj + ik\bn_x\cdot G_k\,\bm
\label{eqn74.1.1}
\end{equation}

For the perfect conductor case, we introduced a single pair $(r_l,q_l)$ of scalar
Debye sources for each component, $\Gamma_l,$ of $bD$ and a single harmonic 1-form
$\bj_{Hl}.$ The currents $(\bj_l,\bm_l)$ are given by
\begin{equation}
  \bj_l=ik(d_{\Gamma_l}R_0r_l-\star_2d_{\Gamma_l}R_0q_l)+\bj_{Hl}\text{ and }
\bm_l=\star_2\bj_l.
\end{equation}
Using
$(\cT^{\pm}_{\bxi}(k)(r,q,\bj_H),\cN^{\pm}_{\bEta}(k)(r,q,\bj_H))$ to denote
these operators restricted to $\cM(\Gamma)\times\cH^1(\Gamma),$ with these
relations, the hybrid system of integral operators is then defined to be
\begin{equation}\label{hybrid0}
  \cQ^{\pm}(k)\left(\begin{matrix} r\\q\\ \bj_H\end{matrix}\right)=
\left(\begin{matrix}
   -G_0\star_2d_{\Gamma}\cT^{\pm}_{\bxi}(k)\\\cN_{\bEta}^{\pm}(k)
\end{matrix}\right) \left(\begin{matrix} r\\q\\\bj_H\end{matrix}\right).
\end{equation}
The range of $\star_2d_{\Gamma}\cT^{\pm}_{\bxi}(k)$ is contained in the space of
functions on $\Gamma$ with mean zero on every component.

As noted in~\cite{EpGr2}, for $\omega\neq 0,$ the nullspace of this system
consists of data such that $\bxi^{\pm}_{t}$ is a harmonic 1-form. We can
therefore append relations of the type given in~\eqref{auxeqn}
or~\eqref{auxeqn2} to get an invertible system. A system of the type
in~\eqref{auxeqn}, however, has serious conditioning problems as $\omega\to 0.$

If we suppose that the incoming data $(\bxi^{\In},\bEta^{\In})$ is a solution
of the \THME[$k$], then an auxiliary equation similar to that
in~\eqref{auxeqn2} can be employed, which does not suffer from conditioning
problems at $k=0.$ We let $\{A_j,B_j:\:j=1,\dots,g\}$ be a basis for
$H_1(\Gamma),$ normalized, as above, so that the $\{A_j\}$ are a basis for
$[H^1_{\dR}(D)]^{\bot}$ and the $\{B_j\},$ a basis for
$[H^1_{\dR}(\Omega)]^{\bot}.$ We augment
\begin{equation}\label{hybrid0.1}
  \cQ^{+}(k)\left(\begin{matrix} r\\q\\ \bj_H\end{matrix}\right)=
\left(\begin{matrix}
   G_0d_{\Gamma}^*\bxi^{\In}_t\\\bEta^{\In}_n
\end{matrix}\right)
\end{equation}
with the equations
\begin{equation}\label{auxeqn2.1}
      \int\limits_{A_j}\star_2\cT^{+}_{\bxi}(k)=\int\limits_{A_j}\bxi^{\In}_t
\end{equation}
\begin{equation}\label{auxeqn2.1.1}
    \frac{1}{k}\int\limits_{B_j}\star_2\cT^{+}_{\bxi}(k)=\frac{1}{k}\int\limits_{B_j}\bxi^{\In}_t.
    \end{equation}
\begin{theorem}
For $\omega\neq 0,$ the null-space of the full
system~\eqref{hybrid0.1},~\eqref{auxeqn2.1}, and~\eqref{auxeqn2.1.1} is
trivial. 
\end{theorem}
\begin{proof} If $(r,q,\bj_H)$ is in the null-space of~\eqref{hybrid0.1},
  then the tangential 1-form $\bxi^+_t\in\cH^1(\Gamma).$ If the integrals
  in~\eqref{auxeqn2.1} and~\eqref{auxeqn2.1.1} vanish, then $\bxi^+_{t}$ is the
  harmonic representative of the trivial cohomology class in
  $H^1_{\dR}(\Gamma),$ and is therefore zero. Hence $\bxi^+_t\equiv 0;$ as
  $\omega\neq 0,$ Theorem 7.1 in~\cite{EpGr2} shows that the data $(r,q,\bj_H)$
  must also vanish.
\end{proof}

At $\omega=0$ the current $\bj=\bj_H,$ and we see that
\begin{equation}\label{eqn888.1.1}
\begin{split}
\bxi^{\pm}_t
&= \frac{\pm \bj_H}{2}-
 d_{\Gamma}G_0r +\star_2K_4\star_2\bj_H\\
[\star_3\bEta^{\pm}]_t
&= \frac{\pm \star_2\bj_H}{2}-
 d_{\Gamma}G_0q -\star_2 K_4\bj_H.
\end{split}
\end{equation}
Since the cycles $\{B_j\}$ span $[H^1_{\dR}(\Omega)]^{\bot},$ there are smooth
surfaces $\{S_j\}$ contained in $\Omega$ so that $b S_j =B_j.$ Thus, using
Stokes theorem and the equations $d\bxi^+=ik\bEta^+,$
$d\bxi^{\In}=ik\bEta^{\In},$ we can rewrite the conditions in~\eqref{auxeqn2.1.1} as
\begin{equation}\label{auxeqn2.2}
  \int\limits_{S_j}\cN^{+}_{\bEta}(k)(r,q,\bj_{H})=\int\limits_{S_j}\bEta^{\In}.
\end{equation}
The advantage of this formulation is that it has an obvious smooth limit as
$k\to 0.$ In fact, replacing~\eqref{auxeqn2.1.1} with these equivalent
conditions allows us to verify that the augmented system of
equations~\eqref{hybrid0.1},~\eqref{auxeqn2.1}, and~\eqref{auxeqn2.2} does not
suffer from low frequency breakdown.
\begin{theorem} If $(r,q,\bj_H)\in\cM_{\Gamma}\times\cH^1(\Gamma)$ defines
  solutions $(\bxi^{\pm},\bEta^{\pm})$ to \THME[$0$] such that
  \begin{equation}
 \cQ^{+}(0)\left(\begin{matrix} r\\q\\
     \bj_H\end{matrix}\right)=\left(\begin{matrix} 0\\0\end{matrix}\right),   
  \end{equation}
  and the integrals on the left hand sides in~\eqref{auxeqn2.1}
  and~\eqref{auxeqn2.2} vanish, then $r=q=\bj_H=0.$
\end{theorem}

\begin{proof}
 
If, for $k=0,$ $\bxi^+$ comes from data
$(r,q,\bj_H)\in\cM_{\Gamma,0}\times\cH^1(\Gamma)$ that is in the null-space
of~\eqref{hybrid0.1}, then, also using the fact that $d\bxi^+=0,$ we see that
$\bxi^+_t$ satisfies
\begin{equation}
  d_{\Gamma}\bxi^{+}_t=0=d_{\Gamma}^*\bxi^{+}_t.
\end{equation}
Thus $\bxi^+_t$ is a harmonic 1-form that belongs to the image of
$H^1_{\dR}(\Omega)\hookrightarrow H^1_{\dR}(b\Omega).$ If the integrals in the
first line of~\eqref{auxeqn2.1} are also $0,$ then this implies that the
cohomology class represented by $\bxi^+_t$ is trivial and therefore
$\bxi^+_t=0.$ The integral of $i_{\bn}\bxi^+$ vanishes over each component of
$\Gamma$ and therefore the standard uniqueness result, see~\cite{ColtonKress},
for outgoing harmonic 1-forms in $\Omega$ shows that $\bxi^+=0$.

If $\bEta^+$ is in the null-space of~\eqref{hybrid0.1}, then the normal
component, $\bEta^+\restrictedto_{b\Omega}=0.$ This implies that $\bEta^+$ is a
harmonic Neumann field, which therefore is determined by its class in
$H^2_{\dR}(\Omega,b\Omega).$ As the surfaces $\{S_j\}$ define a basis for
$H_2(\Omega,b\Omega)=[H^2_{\dR}(\Omega,b\Omega)]',$ this class is, in turn, specified by the values of
the integrals in~\eqref{auxeqn2.2}. Hence if these integrals are also zero,
then $\bEta^+= 0$ in $\Omega$ as well.

This shows that if there is data $(r,q,\bj_H)$ for which both
\begin{equation}
  \cQ^{+}(0)\left(\begin{matrix} r\\q\\ \bj_H\end{matrix}\right)=
\left(\begin{matrix}0\\0\end{matrix}\right),
\end{equation}
and the integrals in the auxiliary conditions,~\eqref{auxeqn2.1}
and~\eqref{auxeqn2.2}, vanish, then $(\bxi^+,\bEta^+)\equiv (0,0).$ From the 
jump relations implicit in~\eqref{eqn888.1.1} we conclude that
\begin{equation}
  \bxi^-_{t}=-\bj_H\text{ and }[\star_3\bEta^-]_t=-\star_2\bj_H.
\end{equation}
Since $d\bxi^-=d\star_3\bEta^-=0$ this implies that both $\bj_H$ and
$\star_2\bj_H$ must belong to $\Im(H^1_{\dR}(D)\hookrightarrow H^1_{\dR}(b\Omega)).$
This is a Lagrangian subspace with respect to the wedge-product pairing, and
therefore
\begin{equation}
  \int\limits_{b\Omega}\bj_H\wedge\star_2\bj_H=0.
\end{equation}
Hence  $\bj_H=0$ as well. The fact that $(\bxi^+,\bEta^+)=(0,0)$ now
implies that their normal components vanish. Coupled with the vanishing of
$\bj_H$ we see that this shows:
\begin{equation}
  \left(\frac{-\Id}{2}+K_0(0)\right)q= \left(\frac{-\Id}{2}+K_0(0)\right)r=0.
\end{equation}
It is a classical result, see~\cite{ColtonKress,Nedelec}, that these operators are
invertible and therefore $(r,q)=(0,0).$ This completes the proof of the theorem
and demonstrates that the hybrid equations along with the auxiliary conditions
in~\eqref{auxeqn2.1} and~\eqref{auxeqn2.2} do not suffer from low frequency
breakdown.
\end{proof}

For any harmonic 1-form $\bj_H\in\cH^1(\Gamma),$ we see that, at $\omega=0,$
\begin{equation}\label{eqn119.1}
\begin{split}
  \frac{\bj_H}{2}+\star_2K_4\star_2\bj_H&\in
  \Im(H^1_{\dR}(\Omega)\hookrightarrow H^1_{\dR}(\Gamma))\\
 \frac{\star_2\bj_H}{2}-\star_2K_4\bj_H&\in \Im(H^1_{\dR}(\Omega)\hookrightarrow
 H^1_{\dR}(\Gamma)).
\end{split}
\end{equation}
So these inclusion maps have rank equal to $\frac 12\dim H^1_{\dR}(\Gamma).$
The system~\eqref{hybrid0.1}, \eqref{auxeqn2.1}, and~\eqref{auxeqn2.2} splits into
two almost uncoupled systems, one for $\bxi^+$ as a function of
$d_{\Gamma}^*\bxi^+_t$ and the integrals in~\eqref{auxeqn2.1}, and the other a
system for $\bEta^+$ in terms of $\bEta\restrictedto_{\Gamma}$ and the
integrals in~\eqref{auxeqn2.2}. These systems are coupled only through $\bj_H,$
which appears in both rows of~\eqref{hybrid0.1}.  

The coupling can be effectively removed by choosing a basis
$\{\psi_j:\:j=1,\dots,2g\}$ for $\cH^1(\Gamma),$ so that
$\{\psi_{1},\dots,\psi_{g}\}$ spans the nullspace of the operator
$\star_2\bj_H-2\star_2K_4\bj_H,$ as a map from $\cH^1(\Gamma)$ to
$H^1_{\dR}(\Gamma),$ and $\{\psi_{1+g},\dots,\psi_{2g}\}$ spans the nullspace
of $(\bj_H+2\star_2K_4\star_2\bj_H):\cH^1(\Gamma)\to H^1_{\dR}(\Gamma).$ The
uniqueness theorem just proved, and the invertibility of $r\mapsto G_0r$ on
$\CI(\Gamma),$ show that these subspaces are complementary, and therefore
\begin{equation}
  \left\{\left(\frac{\Id}{2}+\star_2K_4\star_2\right)\psi_j:\:j=1,\dots,g\right\}
\end{equation}
spans the image of $H^1_{\dR}(\Omega)$ in $H^1_{\dR}(\Gamma).$
Using the conditions in~\eqref{auxeqn2.1} we can obtain
coefficients $\{\alpha_j\}$ so that
\begin{equation}
 \kappa= \bxi^{\In}_t+\sum_{j=1}^g\alpha_j\left(\frac{\Id}{2}+\star_2K_4\star_2\right)\psi_j
\end{equation}
is trivial in $H^1_{\dR}(\Gamma).$ The Debye source $r$ is then found by
solving the equation:
\begin{equation}
  -G_0d_{\Gamma}^*d_{\Gamma}G_0r=G_0d_{\Gamma}^*\kappa.
\end{equation}
An analogous discussion applies to the equations for the magnetic field.

For numerical purposes it is possible to split the difference in the integrals
over the $B$-cycles. The
integrals of the data $\bEta^{\In}$ are computed over the spanning surfaces
$\{S_j\}.$ The integrals of $\bxi^{+}_{t}$ over cycles $\{B_j\}$ can be
accurately computed to order $k,$ by taking advantage of the  fact that when
$k=0,$ we have
\begin{equation}
  \int\limits_{B_j}\star_2\cT^{+}_{\bxi}(0)
=\int\limits_{S_j}d\bxi^{+}=0\text{ for }j=1,\dots,g.
\end{equation}
We can therefore rewrite the  equations in~\eqref{auxeqn2.1.1} as
\begin{equation}\label{auxeqn2.2.1}
  \frac{1}{k}\int\limits_{B_j}
\star_2[\cT^{+}_{\bxi}(k)-\cT^{+}_{\bxi}(0)]=i\int\limits_{S_j}\bEta^{\In}.
\end{equation}
The accuracy of the calculation on the left hand side is retained by observing
that, as $|k|$ is assumed to be small and $|x-y|$ is bounded on $\Gamma,$ we can employ:
\begin{equation}
  \frac{1}{k}[g_k(x,y)-g_0(x,y)]= i \sum_{j=0}^{\infty}\frac{(ik|x-y|)^j}{4\pi(j+1)!}
\end{equation}
to avoid having to explicitly divide by $k,$ which thereby avoids catastrophic cancellation.

To use either~\eqref{auxeqn2.2}, or~\eqref{auxeqn2.2.1},  requires finding surfaces $\{S_j\}$ that
span the $B$-cycles, i.e.
\begin{equation}
  bS_j=B_j\text{ for }j=1,\dots,g,
\end{equation}
and performing certain 1- and 2-dimensional integrals. 
The use of  \eqref{auxeqn2.2.1} has the advantage that the integral over the (artificial) spanning surface
${S_j}$ involves a known and typically smooth incoming field, while the complicated integrand
$[\cT^{+}_{\bxi}(k)-\cT^{+}_{\bxi}(0)]$ needs only be computed over $B_j$ which lies on the surface
$bD$ itself. 


\section{The Dielectric Problem on a Torus of Revolution}\label{sec6}

In the case where $\Gamma$ can be described as a surface of revolution, we can
represent the unknown charges $r_1$, $r_0$, $q_1$, and $q_0$ in terms of their
Fourier expansions in the azimuthal variable. 
More precisely, we assume that the surface $\Gamma$ is given by 

\begin{align*}
x(t,\theta) &= \rho(t) \cos \theta \\
y(t,\theta) &= \rho(t) \sin \theta \\
z(t,\theta) &= z(t),
\end{align*}
where the generating curve $\gamma(t)  = (\rho(t),z(t))$ is smooth, with 
$t \in[0,L]$ and $\theta \in [0,2\pi]$
(Fig. \ref{surfrevfig}). It is straightforward to verify that 
\[  \bj_{H_1}(t,\theta) = \left(  \frac{-\sin\theta}{\rho(t)},  \frac{\cos\theta}{\rho(t)}, 0 \right) \quad{\rm and}\quad
  \bj_{H_2}(t,\theta) = \left(\cos\theta, \sin\theta, \frac{z'(t)}{\rho(t)} \right)
\]
form a basis for the harmonic 1-forms on the surface $\Gamma$.

\begin{figure}[hh]
\centering
{\epsfig{file=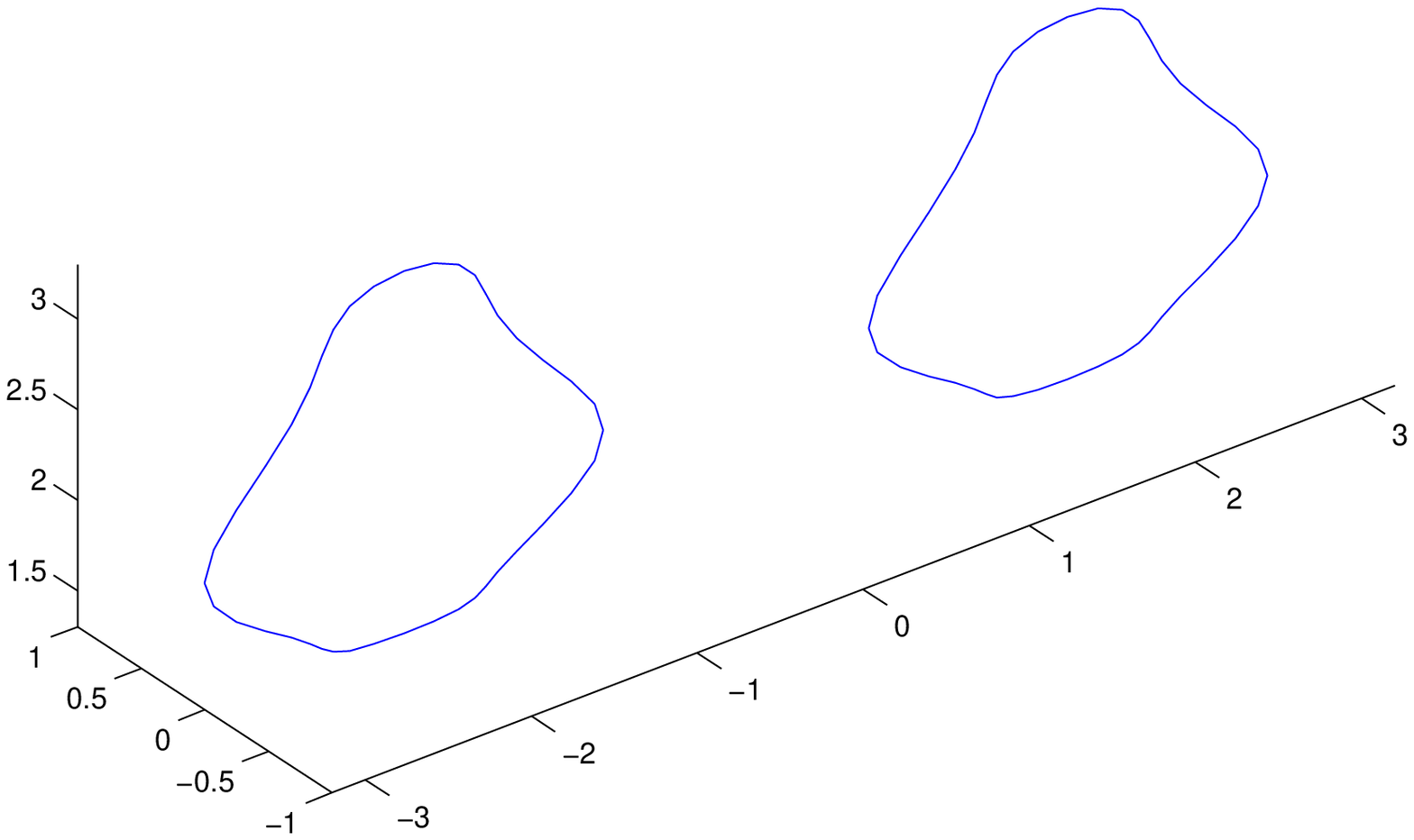,width=6cm}}
{\epsfig{file=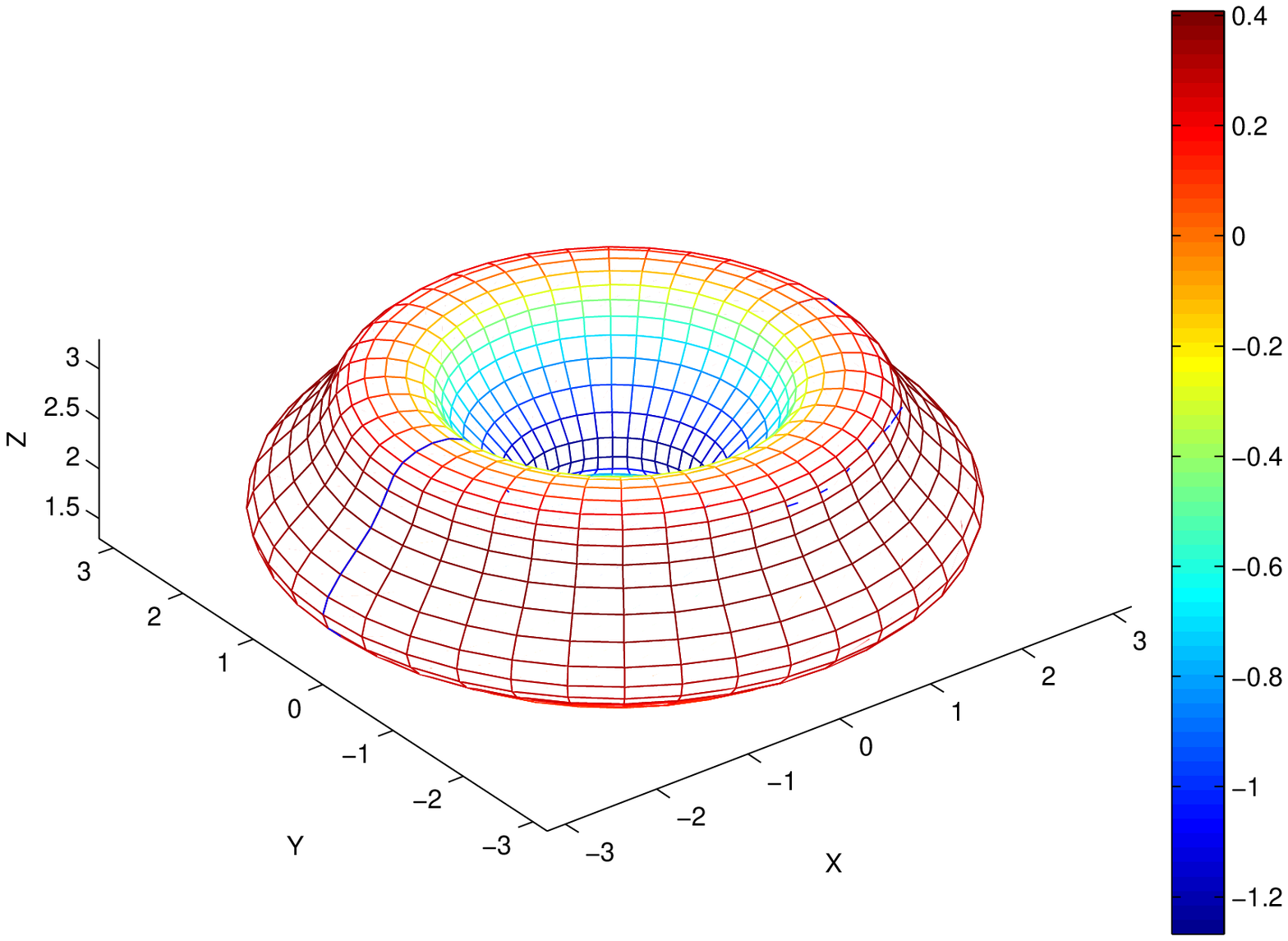,width=6cm}}
\caption{A surface of revolution $\Gamma$ (of genus 1),
defined by a generating curve $\gamma(t) = (\rho(t),z(t))$ which is rotated about the $z$-axis.
We let $\rho(t) = 2 + (1 + 0.2 \, \cos(4t)) \, \cos(t),
z(t) = 2 + (1 + 0.3 \, \sin(4t)) \, \sin(t)$.
On the left is plotted the intersection of the surface with the $xz$-plane. On the right, the
the  $z$-component of the harmonic 1-form ${\bj_H}_2$  is plotted as a
color value on the surface.}
\label{surfrevfig}
\end{figure}

In this setting, we may expand each scalar Debye source $r_l$ and $q_l$, for 
$l = 0,1$,  as a Fourier series in the $\theta$ direction:
\begin{align*}
r_l(t,\theta) &=  \sum_{n=-\infty}^\infty r_{l,n}(t) e^{i n \theta} \\
q_l(t,\theta) &=  \sum_{n=-\infty}^\infty q_{l,n}(t) e^{i n \theta} .
\end{align*}

Using this representation for the scalar sources leads to our
electric and magnetic current-like variables via 
equations \eqref{j1r1r0} and \eqref{m1q1q0}:
\begin{equation}
\begin{split}
\bj_l(t,\theta) &= \alpha_j \bj_{H_1} + \beta_j \bj_{H_2} +
\sum_{n=-\infty}^\infty \left(j^\tau_{l,n}(t) \boldsymbol \tau +
j^\theta_{l,n}(t) \boldsymbol \theta \right) e^{in\theta}, \\
\bm_l(t,\theta) &= \alpha_m \bm_{H_1} + \beta_m \bm_{H_2} +
\sum_{n=-\infty}^{\infty} \left(m^\tau_{l,n}(t) \boldsymbol \tau +
m^\theta_{l,n}(t) \boldsymbol \theta \right) e^{in\theta},
\end{split}
\end{equation}
where $\boldsymbol \tau$ and $\boldsymbol \theta$ are globally defined
orthonormal unit co-vectors defined
along $\Gamma.$
Recall from above that the harmonic 1-forms ${\bj_H}_1$ and ${\bj_H}_2$ 
(${\bm_H}_1$ and ${\bm_H}_2$)
are constant in the azimuthal direction, so that
they couple only  to the purely axisymmetric ($n=0$) mode.

In short, the linear system \eqref{tngeqns}, \eqref{nrmeqns}, \eqref{auxeqn2}
can be solved for each azimuthal mode separately.  For nonzero azimuthal modes,
where the harmonic 1-forms play no role, only equations \eqref{tngeqns},
\eqref{nrmeqns} are needed to determine $r_{l,n}(t), q_{l,n}(t)$ for
$l=0,1$. For each mode the set of equations corresponds to a Fredholm system of
equations of the second kind along the generating curve.

In a companion paper \cite{EGO2}, we have developed a high-order accurate
solver for scattering from closed surfaces of revolution,
with a detailed description of the 
full algorithm. Here, we simply note that the method of \cite{EGO2} assumes an
equi-spaced discretization of $\gamma(t)$, and uses a {\em pseudospectral}
approach for applying (or inverting) the surface Laplacian,
the surface gradient, and the surface divergence. That is to say, differentiation
is carried out in the transform domain, while multiplication by a variable
coefficient function is carried out pointwise along $\gamma(t)$.
The principal value and weakly singular integrals which 
appear are computed using generalized Gaussian quadrature rules
\cite{Alpert}  with eighth or sixteenth order accuracy.

One issue that requires some care is the mean zero condition on 
the scalar sources. We showed, above, that  the
integral equation is invertible as a map from mean zero functions to mean zero
functions. This condition is automatically satisfied for non-zero azimuthal
modes (which necessarily integrate to zero). For the zero mode, one can
either solve the integral equation iteratively, relying on the fact that the operator
projects onto mean zero functions {\em or} add a rank-one modification to the 
system matrix to enforce invertibility of the discretized system
and solve it directly. We have chosen
the latter approach, described in more detail in \cite{EGO2}.

\begin{remark}
In the case of a perfect conductor, there are only
two scalar source functions $(r(t,\theta),q(t,\theta))$,
expanded as above, with the 
electric and magnetic current-like  variables computed from \eqref{jdefpc} and the condition
 $\bm=\star_2\bj$.
\end{remark}

\begin{remark}
In the dielectric case, we still need to specify the choice of a 
clutching map. Our analysis above showed that a sufficient condition for
uniqueness that avoids low frequency breakdown is to set
$\cU=\Id$ on $\cH^1(bD)$. Fig. \ref{clutchfig} displays the condition
number of the linear system for various choices of $\omega$ and
$\cU$.
\end{remark}

\begin{remark}
The second kind Fredholm system derived earlier
for scattering from a dielectric is free from low frequency breakdown,
as well as spurious resonances. An illustration of this is contained
in Fig. \ref{accuracyfig}, which shows the relative $l^2$ error
in the electric and
magnetic fields for the $n=0$ mode in
the exterior and interior  of the object in Fig. \ref{surfrevfig}.
The surface was discretized at 200 points in the $t$ direction 
using a sixteenth order hybrid
Gauss-trapezoidal quadrature rule (see \cite{Alpert}). In order to calculate
the accuracy of the scheme, two extra surfaces of revolution were constructed - one
inside the scatterer and one outside. Smooth generalized Debye sources and harmonic
vector fields were specified on each, and the electric and magnetic fields generated
with the appropriate material parameters
were evaluated on the surface of the scatterer.
The source surface inside the scatterer
gives rise to a field which is valid in the exterior of the
scatterer, and the source surface outside the scatterer
gives rise to a field which is valid in the interior of the scatterer.
The tangential and normal components of the difference of
these fields were taken as the boundary data. Using this boundary data,
the generalized Debye sources and harmonic vector fields on the scatterer were calculated
using the integral equation formulation for the dielectric.
Lastly, the scattered field was evaluated and compared with the known fields
generated by the two source surfaces.

Note that the error becomes even smaller at very small
frequencies - this is because discretization errors that are
introduced in the integral equation in terms that are
$\mathcal O(\omega)$ and $\mathcal O(\omega^2)$ are suppressed
as $\omega \to 0$. Only the accuracy for the $n=0$ mode is plotted.
The relative error, as a function of the frequency $\omega,$ in the
$n=0$ mode is indicative of that in a fully reconstructed scattering problem
(i.e., after reassembling the orthogonal, azimuthally decoupled, field components).
\end{remark}

\begin{figure}[hh]
\centering
{\epsfig{file=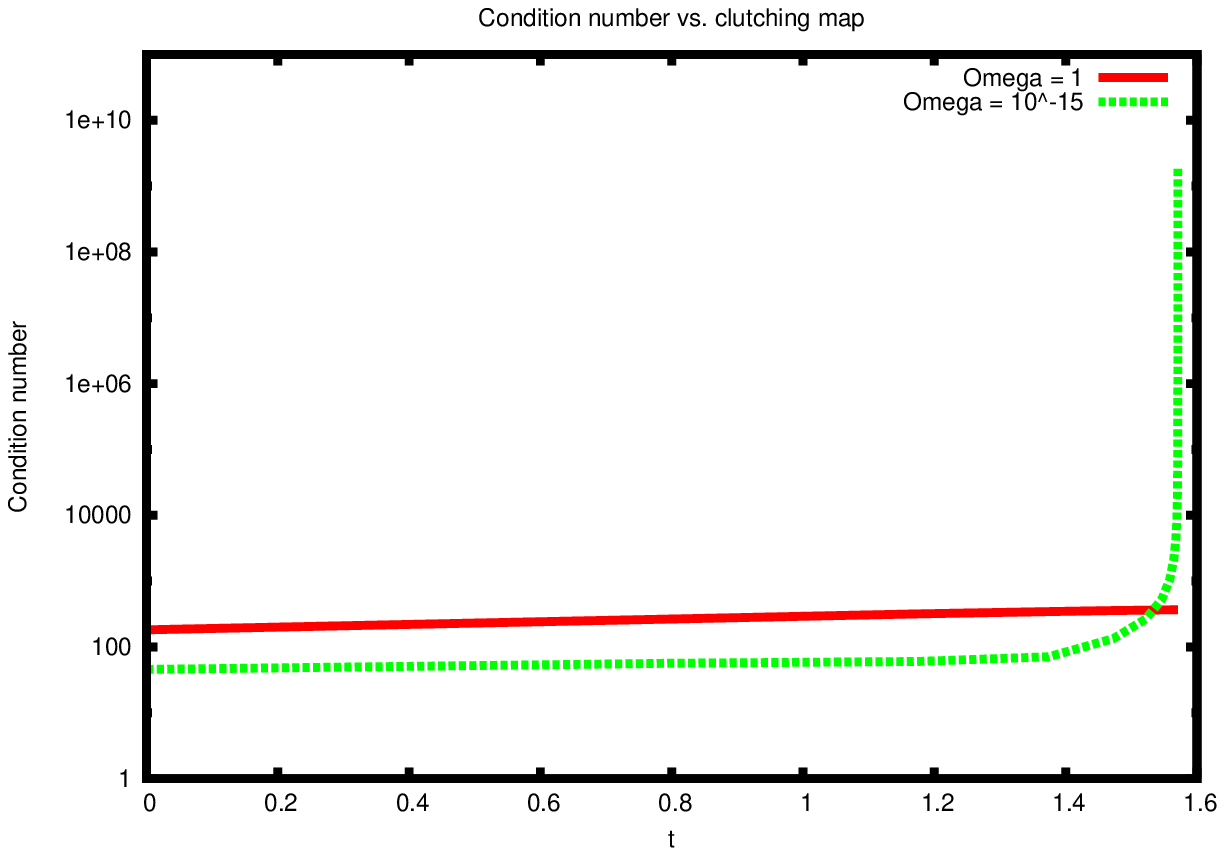,width=6cm}}
{\epsfig{file=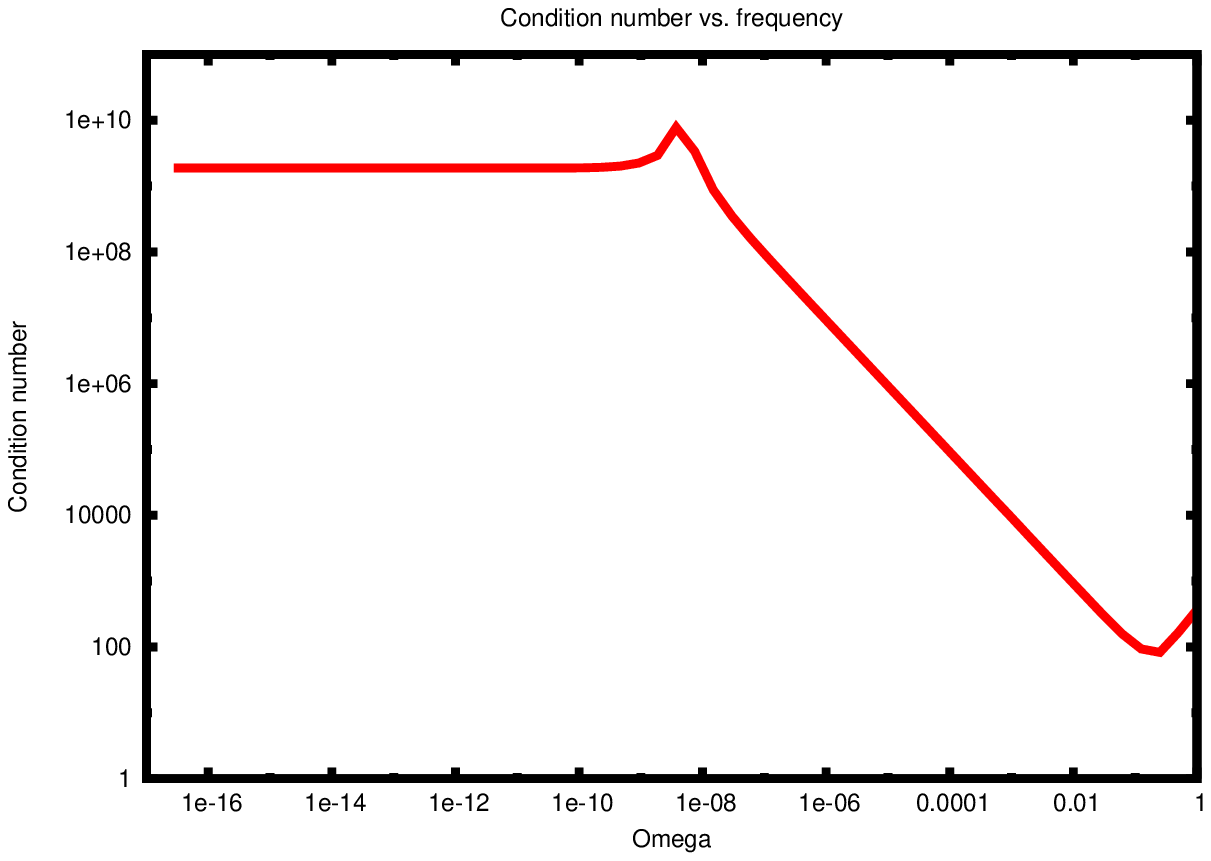,width=6cm}}
\caption{The plots above show the conditioning of 
the linear system for the purely axisymmetric mode depending on
frequency and the choice of clutching map for the surface shown in 
Fig. \ref{surfrevfig}. On $\cH^1(bD)$
we define $\cU=~\cos~t~\Id~+~\sin~t~\star_2$.
On the left is plotted the condition number
for various values of $t$. Note that only when $t$ is near $\pi/2$
and $\omega$ is near zero is the system ill conditioned. On
the right is plotted the dependence of the condition number on $\omega$
when $\cU=\star_2$. In both plots $\epsilon_1=1.30$, $\mu_1=0.83$,
$\epsilon_0=0.90$, and $\mu_1=1.10$.}
\label{clutchfig}
\end{figure}

\begin{figure}[hh]
\centering
{\epsfig{file=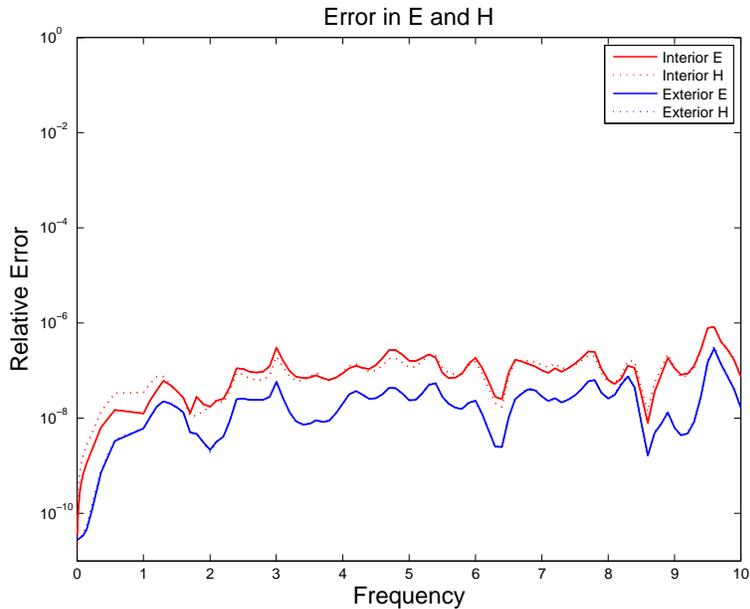,width=10cm}}
\caption{The plot above shows the accuracy obtained in the exterior
and interior electric and magnetic fields at various frequencies
for the purely axisymmetric mode.
The material parameters were chosen to be 
$\epsilon_1=1.30$, $\mu_1=0.83$,
$\epsilon_0=0.90$, and $\mu_1=1.10$.}
\label{accuracyfig}
\end{figure}







\section{Summary}

We have developed a new integral representation for the solution of the time
harmonic Maxwell equations in media with piecewise constant dielectric
permittivity and magnetic permeability in $\bbR^3.$ In the simply connected
case, we rely on four scalar densities (charge-like variables) rather than
surface vector fields as our unknowns, and obtain a system of Fredholm equation
of the second kind.  In the multiply connected case, we supplement these
unknowns with a basis for the surface harmonic 1-forms.  The principal
advantage of our approach is that it avoids the low frequency breakdown
inherent in the classical method, due to M\"{u}ller.  Some subtlety arises in
the selection of the interior and exterior representations. In M\"{u}ller's
equation, the currents used to represent the interior and exterior fields are
chosen to be scalar multiples of each other.  In our case, we use a Hodge
decomposition of the 1-forms and have introduced the {\em clutching map} in
Section 3 to relate interior and exterior variables. We show that for stability
in the zero frequency limit, the clutching map should act differently on the
harmonic 1-forms and on their orthogonal complement.

A disadvantage of our approach is that we must
construct current-like surface 1-forms through a non-local
operator. In the present paper, we do this by inverting the 
Laplace-Beltrami operator, but other (spectrally equivalent) 
procedures can be applied that do not require the solution of a 
linear system of equations \cite{EGGK1,EGO2}. 

{\small \bibliographystyle{siam}{\bibliography{alla-k}}}
\appendix

\section{The Boundary Operators}\label{bdryops}
We now describe the limits of the tangential components of $\bxi$ and $\bEta$
along $\Gamma.$ Along $\Gamma$ we can introduce an adapted local basis of
orthonormal 1-forms, $\{\omega_1,\omega_2,\nu\},$ and $\bn$ the outward unit
normal vector. In this basis we have
$\bxi^{\pm}=a_{\pm}\omega_1+b_{\pm}\omega_2+c_{\pm}\nu.$ The tangential part of
$\bxi^{\pm}$ at $x\in\Gamma$ is the 1-form $\bxi^{\pm}(x)$ restricted to
directions tangent to $\Gamma,$ i.e., $T_x\Gamma.$ In terms of components along
$\Gamma$ we identify the tangential part with
\begin{equation}
\bxi^{\pm}_t=a_{\pm}\omega_1+b_{\pm}\omega_2.
\end{equation}
If the 2-form, $\bEta^{\pm}=e_{\pm}\omega_1\wedge\nu+
f_{\pm}\omega_2\wedge\nu+g_{\pm}\omega_1\wedge\omega_2,$ then the tangential
components of $\bEta^{\pm}$ along $\Gamma$ are identified with the tangential
components of the one-form $i_{\bn}\bEta^{\pm}$ wedged with $-\nu:$
\begin{equation}
\bEta^{\pm}_t=e_{\pm}\omega_1\wedge\nu+f_{\pm}\omega_2\wedge\nu.
\end{equation}
The normal component of $\bxi$ is $i_{\bn}\bxi$ and that of $\bEta$ is simply
$\bEta\restrictedto_{\Gamma}.$

The various boundary operators are given by:
\begin{align*}
K_0[r](\bx) &= \int_\Gamma \frac{\partial g_k}{\partial n_{\bx}}(\bx - \by) \,
r(\by) \, dA(\by) \\
K_1[r](\bx_0) &= \star_2 d_{\Gamma} \int_\Gamma g_k(\bx - \by) \, 
r(\by) \, dA(\by) \, ;
\end{align*}
$K_0$ is an operator of order $-1$ and $K_1$ is an operator of order $0.$

\begin{align*}
K_{2,n} [\bj](\bx) &= i_{\bn}\left[\int_\Gamma  g_k(\bx - \by) \, 
\bj(\by)\cdot d\bx \, dA(\by)\right] \, , \\
K_{2,t} [\bj](\bx) &= \star_2\left[\int_\Gamma   g_k(\bx - \by) \, 
\bj(\by)\cdot d\bx \, dA(\by)\right]_{t} \, .
\end{align*}
These are operators of order $-1.$
\begin{align*}
K_3[\bj](\bx) &= \int_\Gamma 
 d_{\bx} g_k(\bx - \by) \cdot (\bj(\by) \times \bn(\bx))   
\, dA(\by)  \, , \\
K_4[\bj](\bx) &= \int_\Gamma 
\left[ d_{\bx} g_k(\bx - \by) \, (\bj(\by) \cdot \bn(\bx)) -  
\frac{\partial g_k}{\partial n_{\bx}} (\bx - \by) \, \bj(\by)\cdot d\bx \right] 
\, dA(\by) \, .
\end{align*}
$K_3$ is an operator of order $0$ and $K_4$ is an operator of order $-1.$

\end{document}